\documentclass[12pt,reqno]{amsart}

\usepackage{amsmath}
\usepackage{latexsym}
\usepackage{amssymb}
\usepackage{amsfonts}
\usepackage{hyperref}

\usepackage{color}

\usepackage{graphics}

\renewcommand{\proof}{\par\noindent{\it Proof.\ \ }}
\def\qed{\ifmmode\square\else\nolinebreak\hfill
$\square$\fi\par\vskip12pt}

\renewcommand{\proof}{\par\noindent{\it Proof.\ \ }}
\def\qed{\ifmmode\square\else\nolinebreak\hfill
$\square$\fi\par\vskip12pt}

\def\l{\langle} \def\r{\rangle}

 \def\ZZ{\mathbb Z}
\def\GG{{\mathcal G}}

\def\BB{{\mathcal B}} \def\CC{{\mathcal C}}

\def\PG{{\rm PG}}  \def\Cos{{\sf Cos}}

\def\D{{\rm D}} \def\S{{\rm S}} \def\G{{\sf G}}
 \def\M{\mathrm{M}} \def\soc{{\sf soc}}
 \def\C{{\bf C}} \def\N{{\bf N}}
  \def\mod{{\sf mod~}}

 \def\Aut{{\sf Aut}}

\def\and{{\sf and}}

\def\Out{{\sf Out}}  \def\K{{\sf K}}

\def\Ga{{\it \Gamma}}
\def\Sig{{\it\Sigma}}
\def\Del{{\it\Delta}}
\def\Ome{{\it\Omega}}

\def\a{\alpha} \def\b{\beta} \def\g{\gamma}

 \def\PSp{{\rm PSp}}
 
\def\GammaL{{\rm \Gamma L}}
\def\AGammaL{{\rm A\Gamma L}}
\def\PGammaL{{\rm P\Gamma L}} \def\PSigmaL{{\rm P\Sigma L}}
\def\ASigmaL{{\rm A\Sigma L}}
\def\A{{\rm A}}\def\Sym{{\rm Sym}}

\def\PSL{{\rm PSL}}\def\PGL{{\rm PGL}}
\def\GL{{\rm GL}} \def\SL{{\rm SL}} 
\def\AGL{{\rm AGL}}

 \def\PSU{{\rm PSU}}

  \def\D{{\rm D}} \def\G{{\rm G}}

\newtheorem{theorem}{Theorem}[section]
\newtheorem{lemma}[theorem]{Lemma}%
\newtheorem{proposition}[theorem]{Proposition}%
\newtheorem{example}[theorem]{Example}%

    \def\ZZ{\mathbb Z}
\def\BB{{\mathcal B}} \def\CC{{\mathcal C}}

\begin{document}

\title[Arc-transitive bidihedrants]
{Quasiprimitive groups with a biregular dihedral subgroup,
and arc-transitive bidihedrants}
\thanks{1991 MR Subject Classification 20B15, 20B30, 05C25.}
\thanks{This paper was partially supported by the National Natural Science Foundation of China (11961076, 12071023, 12161141005).}

{\author[Pan]{Jiangmin Pan}
\address{Jiangmin Pan\\
School of Statistics and mathematics\\
Yunnan University of Finance and Economics\\
Kunming, Yunnan 650091\\
P. R. China}
\email{jmpan@ynu.edu.cn}}

\author[Yin]{Fu-Gang Yin}
\address{Fu-Gang Yin\\
School of Mathematics and Statistics\\
Central South University\\
Changsha 410083, Hunan\\ 
P. R. China}
\email{18118010@bjtu.edu.cn}

\author[Zhou]{Jin-Xin Zhou}
\address{Jin-Xin Zhou\\
School of Mathematics and Statistics\\ 
Beijing Jiaotong University\\
Beijing 100044\\
P. R. China}
\email{jxzhou@bjtu.edu.cn}

 \curraddr{}


\maketitle

\begin{abstract}
A semiregular permutation group on a set $\Ome$ is called {\em bi-regular} if it has two orbits. A classification is given of quasiprimitive permutation groups with a biregular dihedral subgroup. This is then used to characterize the family of arc-transitive graphs whose automorphism groups containing a bi-regular dihedral subgroup. We first show that every such graph is a normal $r$-cover of an arc-transitive graph whose automorphism group is either quasiprimitive or bi-quasiprimitive on its vertices, and then classify all such quasiprimitive or bi-quasiprimitive arc-transitive graphs.
\end{abstract}

\qquad {\textsc k}{\scriptsize \textsc {eywords.}} {\footnotesize quasiprimitive group, bi-regular, bidihedrant, arc-transitive graph}

\section{Introduction}

Let $G$ be a permutation group on a set $\Omega$. For a point $\a\in\Omega$, denote by $G_\a$ the stabilizer of $\a$ in $G$. We say that $G$ is {\em semiregular} on $\Omega$ if $G_\a=1$ for every $\a\in \Omega$, and {\em regular} if $G$ is transitive and semiregular, and {\em bi-regular} if $G$ is semiregular and has exactly two orbits on $\Ome$.

Let $G\le\Sym(\Ome)$ be a transitive permutation group. If the only partitions of $\Ome$ preserved by $G$ are either the singleton subsets or the whole of $\Ome$, then $G$ is said to be {\em primitive} on $\Omega$. If every non-trivial normal subgroup of $G$ is transitive on $\Omega$, then we say that $G$ is {\em quasiprimitive} on $\Omega$. Note that every primitive group is quasiprimitive.

Primitive permutation groups containing a certain subgroup have been investigated for over a century. In 1900, Burnside \cite{Burnside} proved that a primitive permutation group containing a cyclic regular subgroup either is $2$-transitive or has prime degree. Burnside's result was extended by Schur~\cite{Schur} in 1933 to general cyclic regular subgroups, and further extended by Li \cite{Li-Proc03} in 2003 to abelian regular subgroups. Interest in quasiprimitive groups containing a transitive metacyclic subgroup dates back to 1950 when Wielandt~\cite{Wielandt1950} proved that finite primitive groups containing a regular dihedral group are $2$-transitive. Since then, such groups has been investigated by many authors (see, for example, Scott \cite{Scott}, Nagai \cite{Nagai}, Jones \cite{Jones72}, Li \cite{Li-TAMS05}). In 2021, a complete classification of quasiprimitive groups containing a transitive metacyclic subgroup was achieved by Li et al. in \cite{LPX2021}.

Primitive permutation groups containing a permutation with at most two cycles are also essential in certain number theoretical problem~\cite{Fuchs2012}. Motivated by this, M\" uller \cite{Muller13} in 2013 classified primitive permutation groups containing a bi-regular
cyclic subgroup. This was further extended in \cite{GMPS16} by classifying primitive permutation groups containing a permutation having at most four cycles. 
In this paper, we will classify quasiprimitive permutation groups containing a bi-regular dihedral subgroup. The following is our first theorem.

\begin{theorem}\label{Thm-1}
Let $G$ be a quasiprimitive permutation group  containing a bi-regular dihedral subgroup $H\cong \D_{2n}$ with $n\ge 2$. Then one of the following is true:
\begin{itemize}
\item[(1)] $G$ is  primitive, and one of the following holds:
\begin{itemize}
\item[(i)]  $(G,H)=(\AGL_1(8),\D_4), (\AGammaL_1(8),\D_4), (\AGL_3(2),\D_4)$,
$(\ZZ_2^4:\mathrm{SO}_{4}^{+}(2),\D_8)$, $(\ZZ_2^4:\mathrm{SO}_{4}^{-}(2),\D_8)$, $(\ZZ_2^4:\S_6,\D_8)$,
$(\AGL_4(2),\D_8)$ or $(\AGL_5(2),\D_{16})$;
\item[(ii)] $\soc(G)=\A_{4n}$  in its natural action of degree $4n$;
\item[(iii)] $\soc(G)=\PSL_2(q)$ with $q\equiv 3\pmod{4}$ in its natural action of degree $q+1$;
\item[(iv)] $G=\mathrm{M}_{12}$ in its natural action of degree $12$;
\item[(v)] $G=\mathrm{M}_{24}$ in its natural action of degree $24$;

\end{itemize}

\item[(2)] $G$ is imprimitive, and $(G,H,G_\a)$ is one of the following triples:
\begin{itemize}
\item[(i)] $(\A_5,\D_{10},\ZZ_3), (\S_5,\D_{10},\ZZ_3), (\S_5,\D_{10},\ZZ_6)$;
\item[(ii)] $G=\PSL_{2}(p^e).f$, $H=\D_{p^e+1}$ and $G_\a=\ZZ_p^e : \ZZ_{(p^e-1)/4}.f$, where $p^e \equiv 5\pmod{8}$ and $f\mid e$.
\end{itemize}
\end{itemize}
\end{theorem}

One more motivation of this paper is arisen naturally from the study of arc-transitive bi-Cayley graphs
of dihedral groups. Before proceeding, we set some notation and definitions. Let $\Ga$ be a graph. We denote by $V\Ga$, $E\Ga$ and $\Aut\Ga$ the vertex set, edge set, and full automorphism group of $\Ga$, respectively. Let $G\leq\Aut\Ga$. We say that $\G$ is {\em $G$-vertex-transitive\/} or {\em $G$-arc-transitive\/} if $G$ is transitive on the vertices or the order pairs of adjacent vertices, respectively.
When $G=\Aut\Ga$, a $G$-vertex-transitive or $G$-arc-transitive graph $\Ga$ is simply called {\em vertex-transitive\/} or {\em arc-transitive\/}, respectively. Furthermore, a graph $\Ga$ is said to be \emph{$(G,2)$-arc-transitive} if $G$ acts transitively on the set of all \emph{$2$-arcs} of $\Ga$ (that is, the set of all triples $(u,v,w)$ such that $u,v,w\in V\Ga$, $u\neq w$ and $\{u,v\},\{v,w\}\in E\Ga$); and $\Ga$ is \emph{$2$-arc-transitive} if such a group exists.

A graph $\Ga$ is called a {\it Cayley graph} of a group $G$
if its full automorphism $\Aut\Ga$ contains a subgroup
which is isomorphic to $G$
and regular on the vertex set of $\Ga$, while $\Ga$ is called a {\it bi-Cayley graph} of a group $H$
if $\Aut\Ga$ contains a subgroup that is isomorphic to $H$ and semiregular on $V\Ga$ with exactly two orbits.
For convenience, we often call a Cayley graph of a cyclic group or dihedral group a {\it circulant} or {\it dihedrant}, respectively,
and call a bi-Cayley graph of a cyclic group or dihedral group a {\it bi-circulant} or {\it bi-dihedrant}, respectively.

2-Arc-transitive circulants and dihedrants have been classified by Alspach, Conder, Maru\v si\v c and Xu \cite{ACMX96},
and Maru\v si\v c, Malni\v c and Du \cite{DMM08,Maru03}, respectively.
Bi-Cayley graphs can be viewed as a natural generalization of Cayley graphs,
and have many interesting properties, see \cite{CZF20,ZF16} for example;
also there are many well-known graphs which are bi-Cayley graphs but not Cayley graphs,
such as the Petersen graph, the Gray graph and the Hoffman-Singleton graph.
Very recently, Devillers, Giudici and Jin \cite{DGJ22}
characterized arc-transitive bi-circulants.

Our second aim is to characterize arc-transitive bi-dihedrants. The main approach is to use the normal quotient. Let $\Ga$ be a connected graph. Assume that $G\leq\Aut\Ga$ is such that $\Ga$ is $G$-vertex-transitive. Let $N$ be a normal subgroup of $G$ such that $N$ is intransitive on $V\Ga$. The {\em $N$-normal quotient graph\/} of $\Ga$ is defined as the graph $\Ga_N$ with vertices the $N$-orbits in $V\Ga$ and with two distinct $N$-orbits adjacent if there exists an edge in $\Ga$ consisting of one vertex from each of these orbits. The graph $\Ga$ is a {\rm normal $r$-cover} of $\Ga_N$ if for each edge $\{B, B'\}$ of $\Ga_N$ and each $v\in B$, we have $|\Ga(v)\cap B'|=r$. In particular, if $r=1$, then $\Ga$ is called a {\em normal cover} of $\Ga_N$.

Observe that for a $G$-arc-transitive graph $\Ga$ with at least three vertices, we may take a normal subgroup $N$ maximal subject to having at least three orbits on $V\Ga$. Then the $N$-normal quotient $\Ga_N$ is $G/N$-arc-transitive and $G/N$ is either quasiprimitive or bi-quasiprimitive on $V\Ga$. (A transitive permutation group $G\leq\Sym(\Ome)$ is said to be {\em bi-quasiprimitive} if every non-trivial normal subgroup of $G$ has at most two orbits on $\Ome$ and there exists one which has exactly
two orbits on $\Ome$.)  Actually, we can prove that each basic arc-transitive bi-dihedrant is an orbital graph of
a quasiprimitive or bi-quasiprimitive permutation group containing a (bi-)regular cyclic or dihedral subgroup. This enables us to make use of Theorem~\ref{Thm-1} to classify `basic' arc-transitive bi-dihedrants, and so obtain our second theorem.
For convenience, certain graphs appearing in Theorem~\ref{Thm-2} will be introduced in
Section 4.

\begin{theorem}\label{Thm-2}
Every connected arc-transitive bi-dihedrant is a normal $r$-cover of a graph $\Sig$ which is one of the following graphs:
\begin{itemize}
\item[(1)] $\Sig$ is one of the following dihedrant:
\begin{itemize}
\item[(a)] $\K_{4n}$, $\K_{2n,2n}$, $\K_{2n,2n}-2n\K_2$;
\item[(b)]  the  affine polar graphs $\mathrm{VO}^{\pm}_4(2)$ and their complementary graphs $\overline{\mathrm{VO}^{\pm}_4(2)}$;
\item[(c)] $\mathrm{F020A}$, $ \GG_{20}^{(1)}, \GG_{20}^{(2)}, \GG_{20}^{(3)}$;
\item[(d)] $\mathcal{G}_{2,q}$ with $q\equiv 5\pmod{8}$;
\item[(e)] $B(\mathrm{PG}(d-1,q))$ and $B'(\mathrm{PG}(d-1,q))$ with $d\geq 3$;
\item[(f)]  $\mathcal{G}_{d,q}^{(1)}$ or $ \mathcal{G}_{d,q}^{(2)}$  with $d\geq 3$ odd and $q$ odd.
\end{itemize}
\item[(2)] $\Sig$ is an arc-transitive bi-circulant or an arc-transitive circulant, as listed in \cite[Theorem 1.1]{DGJ22};
\item[(3)] $\Sig$ is an arc-transitive dihedrant, as listed in \cite[Theorem 1.2]{Pan12}.
\end{itemize}
\end{theorem}

\section{Preliminaries}

\subsection{Two group theoretical results}
The first result concerns primitive permutation groups with an element having at most four cycles, see \cite[Theorem 1.1]{GMPS16}.

\begin{theorem}\label{Less-4Cyc}
Let $G$ be a finite primitive permutation of degree $n$, and let $g$ be an element of $G$ having cycle lengths
$($counted with multiplicity$)$ $(n_1,\dots,n_N)$ with $N\le 4$.
Then the socle $\soc(G)=T^{\ell}$ for some simple group $T$ and integer $\ell$, and one of the following holds:

\begin{itemize}
\item[(a)] $\soc(G) =\A_m^{\ell}$ in its natural product action of degree $m^{\ell} $ with  $\ell \leq 3$;
\item[(b)] $\soc(G) =\PSL_d(q)^{\ell}$ in its natural product action of degree $((q^d-1)/(q-1))^{\ell} $ with  $\ell \leq 3$ and $d\geq 2$;
\item[(c)] $T$, $\ell$ and $|\Omega|$ are in one of the rows of  Tables 3, 4 or 5 of~\cite{GMPS16};
\item[(d)]  $\soc(G)$ is elementary abelian, and $g$ and $G$ are described in Theorems 1.5 and 1.6 of~\cite{GMPS15}.
\end{itemize}
Moreover, Tables 1 and 2 list all of the possible $N$-tuples $(n_1,\dots,n_N)$ arising from parts (a) and (b) respectively.
\end{theorem}

The next is about quasiprimitive permutation groups containing a regular dihedral group.
Such groups were classified in \cite[Theorem~1.5]{Li-TAMS05}, with two examples missed and pointed out on \cite{LSW14}, namely,
the affine groups $\AGammaL(2,4)\cong\ZZ_2^4:\GammaL(2,4)\cong\ZZ_2^4:(\ZZ_3\times\A_5).\ZZ_2$
and $\ASigmaL(2,4)\cong\ZZ_2^4:\S_5$, both containing a regular dihedral subgroup $\D_{16}$.

\begin{theorem}\label{pri-d-gps}
Let $G\le\Sym(\Ome)$ be a quasiprimitive permutation group which
contains a regular dihedral subgroup $H$. Then $H$ is
$2$-transitive, and one of the following holds, where $w\in\Ome$:

\begin{itemize}
\item[(a)] $(G,G_w,H)=(\A_4, \ZZ_3,2)$, $(\S_4,\S_3,2)$, $(\AGL_3(2),\GL_3(2),4)$, $(\AGL_4(2),\GL_4(2),8)$,
$(2^4:\S_5,\S_5,\D_{16})$,  $(\ZZ_2^4:\A_6,\A_6,8)$,  $(2^4:\GammaL_2(4),\D_{16})$ or $(\ZZ_2^4:\A_7,\A_7,8)$;
\item[(b)] $(G,G_w,H)=((\S_{2m},\D_{2m},\S_{2m-1}), \A_{4m},\D_{4m},\A_{4m-1})$;
\item[(c)] $(G,G_w,H)=(\M_{12},\D_{12},\M_{11}), (\M_{24},\D_{24},\M_{23}),(\M_{22}.2,\D_{22},\PSL(3,4).2)$;
\item[(d)] $G=\PSL_2(p^e).o$, $H=\D_{p^e+1}$ and $G_w=\ZZ_p^e:\ZZ_{p^e-1\over2}.o$ where $p^e\equiv3\ (\mod 4)$ and $o\le\Out(\PSL_2(p^f))\cong\ZZ_2\times\ZZ_e$;
\item[(e)] $G=\PGL_2(p^e).\ZZ_f$, $H=\D_{p^e+1}$ and $G_w=(\ZZ_p^e:\ZZ_{p^e-1}).\ZZ_f$ where $p^e\equiv 1 ~(\mod 4)$ and $f\mid e$.
\end{itemize}
\end{theorem}

\subsection{Two technical lemmas}

The first is an observation on subgroups of dihedral groups.

\begin{lemma}\label{Subgroup}
Suppose that $H$ is a dihedral group  and $R_1,\dots,R_t$ are subgroups of $H$
such that $|R_1|=\cdots=|R_t|\ge 3$.
Then $R_1\cap\cdots\cap R_t$ contains a nontrivial normal subgroup of $H$.
\end{lemma}
\proof Suppose $H=\l a\r:\l b\r\cong\D_{2n}$ with $o(a)=n, o(b)=2$ and $a^b=a^{-1}$ and $|R_1|=m$.
If $m=2n$, nothing need to prove. Assume $m$ is a proper divisor of $2n$.
Notice that each subgroup with order $m$ of $H$ has the form:
$$\l a^{n\over m}\r~~with~~m\mid n~~,~~or ~~\l a^{2n\over m}\r:\l a^jb\r~~with~~m~~even~~and~~j~~an~~integer,$$
we  derive that $\l a^{2n\over m}\r\ne 1$ is normal in $H$ and contained in $R_1\cap\cdots\cap R_t$.\qed

A \emph{Singer cycle} of $\GL_d(q)$ is an element of order $q^d-1$, and the subgroup generated by a Singer cycle  is called a \emph{Singer subgroup}.
The following property on Singer subgroups is well-known, see Derek Holt on the Stack Exchange website.
We include a proof here as it is not lengthy.

\begin{proposition}
Any two Singer subgroups of $\GL_d(q)$ are conjugate.
\end{proposition}

\begin{proof}
Let $V=\mathbb{F}_q^d$ and $G=\GL(V)=\GL_d(q)$, and let $H$ be a Singer subgroup  of $G$.
Since $H$ is cyclic and of order $ q^d-1$, it acts regularly on the non-zero vectors of $V$.

Assume that neither $(d,q)=(6,2)$ nor $d=2$ and $q+1$ is a $2$-power.
Then by a well known result of Zsigmondy, there is a prime $r$ that divides $q^d-1$, but does not divide $q^i-1$ for any $i<d$.
Then $H$ contains a Sylow $r$-subgroup $R$ of $G$ and $H\leq \mathbf{C}_G(R)$.
Since $R$ acts irreducibly on $V$, by Schur's lemma, $\mathrm{End}_{\mathbb{F}_q R}(V)$ is a  division ring.
Moreover, $\mathrm{End}_{\mathbb{F}_q R}(V)$ is a field as it is finite.
Notice that $ \mathbf{C}_G(R)$ is a subgroup of the multiplication group of $\mathrm{End}_{\mathbb{F}_q R}(V)$.
Thus, $ \mathbf{C}_G(R)$  is  cyclic.
Then $ \mathbf{C}_G(R)$ acts regularly on the non-zero vectors of $V$, which implies that $\mathbf{C}_G(R)=H$.
Now it follows from Sylow's theorem that any two Singer  subgroups of $\GL_d(q)$ are conjugate.

Assume that $d=2$ and $q+1$ is a $2$-power. Then $H$ contains an irreducible Sylow $2$-subgroup $R$ of $G$. The remaining argument follows a similar pattern as described above.
Now, let's consider the case where $(d,q)=(6,2)$. By computation in {\sc Magma}~\cite{Magma}, $\GL_6(2)$ has only one conjugacy class of cyclic subgroups of order $63$. \hfill\qed
\end{proof}

We know that $\mathbb{F}_{q^d} $ can be viewed as a $d$-dimensional vector space over $\mathbb{F}_q$.
The multiplication of $\mathbb{F}_{q^d}^{*}$ on $\mathbb{F}_{q^d}$ are $\mathbb{F}_q$-linear transformation of $\mathbb{F}_{q^d}$, and this induces an homomorphism $ \psi: \mathbb{F}_{q^d}^{*} \to \GL_d(q)$.
The group $\psi(\mathbb{F}_{q^d}^{*} )\cong \GL_1(q^d) $ is exactly a Singer subgroup of $\GL_d(q)$.
Notice that $\psi(\mathbb{F}_{q}^{*} )$ is exactly the center $Z$ of $\GL_d(q)$, and hence $Z $ is contained in any Singer subgroup.
A Singer subgroup of $\PGL_d(q)$ is the projective image of a Singer subgroup of $\GL_d(q)$.

View $\GL_1(q^d)$ as a Singer subgroup of $\GL_n(q)$.
By Huppert~\cite[p.187, Satz~7.3]{Huppert1967},  $\N_{\GL_d(q)}(\GL_1(q^d))= \GammaL_1(q)\cong \ZZ_{q^d-1}:\ZZ_d$, where $\ZZ_d$ is the  Galois group $\mathrm{Gal}(\mathbb{F}_{q^d}/\mathbb{F}_q)$.
Let $q=p^e$ where $p$ is a prime and $e$ a positive integer.
Then $\GammaL_d(q)=\GL_d(q):\ZZ_e$, where $\ZZ_e=\mathrm{Gal}(\mathbb{F}_{p^e}/\mathbb{F}_p)$.
Moreover, $\N_{\GammaL_d(q)}(\GL_1(q^d))=\GammaL_1(p^{de})=\ZZ_{q^d-1}:\ZZ_{de}$ (see for example~\cite[Proposition~2.2]{PS2016}).
Let $z$ be an element of $\GL_1(q^d)$ such that the order of $z$ is not a divisor of $q^m-1$ for any $1\leq m<n$ and $m\mid  n$.
Then  $\N_{\GL_n(q)}(\langle z\rangle)=\N_{\GL_n(q)}(\GL_1(q^d))$ by~\cite[p.187, Satz~7.3]{Huppert1967}.
Since $\langle z\rangle $ is characteristic in $\GL_1(q^d)$, it follows from $\N_{\GammaL_d(q)}(\GL_1(q^d))=\GammaL_1(p^{de})$ that $\N_{\GammaL_d(q)}(\langle z\rangle) \geq \GammaL_1(p^{de})$.
Notice that $\mathbb{Z}_p^{de}:\N_{\GammaL_d(q)}(\langle z\rangle)\leq \AGL_{de}(p)$ is an affine $2$-transitive permutation group of degree $p^{de}$.
By checking the list of affine  $2$-transitive permutation groups (see for example~\cite{Liebeck-affine}),
we obtain that $\mathbb{Z}_p^{de}:\N_{\GammaL_d(q)}(\langle z\rangle) \leq \mathbb{Z}_p^{de}:\GammaL_1(p^{de})$, and hence $\N_{\GammaL_n(q)}(\langle z\rangle) = \GammaL_1(p^{de})$.


\begin{lemma} \label{lm:mdq}
Let $d\geq 2$ be an integer and $q=p^e$ a prime power.
Let $G= \mathrm{\Gamma L}_1(p^{de})=\langle x\rangle{:}\,\langle y\rangle$ with $\langle x\rangle\cong \ZZ_{p^{de}-1}$,
$\langle y\rangle\cong \ZZ_{de} $ and $ x^y= x^p$, and let $Z=\langle x^{(q^{d}-1)/(q-1)}\rangle$.
\begin{itemize}
\item[(a)] If $d\geq 3$ and $G$ has an element $g$ such that $x^m(x^m)^g \in Z$ for some  $m\in \{1,2,4\}$, then $d=4$, $q=3$ and $m=4$.
\item[(b)] If $d=2$, $q\geq 5$ is odd and $G$ has an element $g$ such that $ x^m(x^m)^g \in Z$ for some $m\in \{1,2\}$, then $g=x^iy^{de/2}$ for some $0\leq i \leq p^{de}-1$.
\end{itemize}
\end{lemma}

\begin{proof}
(a) Computation in {\sc Magma}~\cite{Magma} shows that such an element $g$  exists for $(d,q,m)=(4,3,4)$,
but does not exist for $d=3$ and $q<16$, or $(d,q)=(4,2)$, or $(d,q,m)=(4,3,1)$ or $(4,3,2)$.
We thus assume next that
\begin{equation}\label{eq:dq}
(d,q)\notin \{(3,2),(3,3),(3,4),(3,5),(3,7),(3,8),(3,9),(3,11),(3,13),(4,2),(4,3)\}.
\end{equation}

We may let $g=y^{de/s}$ for some positive integer $s $ such that $s \geq 2$ and $s$ divides $de$.
From  $x^m(x^m)^g\in Z$, we conclude that
\[
x^m(x^m)^g= x^m(x^m)^{y^{\frac{de}{s}}}=x^m  (x^{m \cdot p^{\frac{de}{s}}})=x^{m(1+p^{\frac{de}{s}})} \in Z.
\]
Since $\mathrm{GL}_1(q^d)=\langle x\rangle$ has order $q^d-1$ and the subgroup $Z$ has order $q-1$, we have $Z=\langle x^{(q^d-1)/(q-1)}\rangle$.
From~\eqref{eq:dq} we conclude $p^{de/2}=q^{d/2}> 5$, and so
 \[
(p^{de}-1)-m(1+p^{\frac{de}{s}})\geq (p^{de}-1)-4(1+p^{\frac{de}{2}})=(p^{\frac{de}{2}}-2)^2-9> (5-2)^2-9>0.
\]
This together with $x^{m(1+p^{de/s })} \in Z=\langle x^{(q^d-1)/(q-1)}\rangle$  imply that  there is some $t\in \{1,\ldots,q-1\}$ such that $ m(1+p^{de/s})=t(q^d-1)/(q-1)=t(p^{de}-1)/(p^{e}-1)$.
In particular, we have $4(1+p^{\frac{de}{2}})\geq (p^{de}-1)/(p^{e}-1)$,  which leads to
\[
0\leq 4(1+p^{\frac{de}{2}})(p^{e}-1)-(p^{de}-1)= 4(1+p^{\frac{de}{2}})(p^{e}-1)-( p^{\frac{de}{2}}+1)( p^{\frac{de}{2}}-1)=( p^{\frac{de}{2}}+1)(4p^e-3-p^{\frac{de}{2}}).
\]
Therefore,   $0 \leq  4p^e-3-p^{\frac{de}{2}} $, which is equal to
\begin{equation}\label{eq:4q-3-qd/2}
  q^{\frac{d}{2}}+3-4q \leq 0.
\end{equation}
If $d\leq 3$, then $q\geq 16$ by~\eqref{eq:dq}, and so
$q^{ {d}/{2}}+3-4q \geq q(q^{1/2}-4)+3\geq 3 $, a contradiction.
If $d=4$, then $q\geq 4$ by~\eqref{eq:dq}, which implies $q^{{d}/{2}}+3-4q =(q-2)^2-1>0$, again a contradiction.
If $d\geq 5$, then
\[
q^{\frac{d}{2}}+3-4q =q^2(q^{\frac{d-4}{2}}-1)+ (q-2)^2-1 >2^2(2^{\frac{1}{2}}-1)-1>0,
\]
which is again a contradiction.

(b) Now $Z=\langle x^{q+1}\rangle$.
Let $g=x^iy^{2e/s}$ for some $s\geq 2$ such that $s $ divides $2e$.
Then from  $x^m(x^m)^g\in Z$  we conclude that  $x^{m(1+p^{ {2e}/{s}})} \in Z$, that is, $x^{m(1+q^{ 2/s})} \in Z$.

Assume that $m=1$.
Then from $x^{ 1+q^{ 2/s} }  \in Z=\langle x^{q+1}\rangle$ we see that $s \in \{1,2\}$.
Further, if $s=1$, then $x^{ 1+q^{ 2/s} }=x^{q^2+1}=x^2$ (as $ \langle x\rangle =\ZZ_{q^2-1}$), which is not in $Z=\langle x^{q+1}\rangle$, a contradiction.
Therefore, $s=2$, as required.

Assume now that $m=2$.  Since $q=p^e\geq 5$ and $s\geq 2$, we have
\[
q^2-1-2(1+q^{\frac{2}{s}})\geq q^2-1-2(1+q)=(q-1)^2-4>0.
\]
Thus, from $x^{ 2(1+q^{ 2/s})}  \in Z=\langle x^{q+1}\rangle$ we conclude that
\[ 2(1+q^{\frac{2}{s}})=t(q+1) \text{ for some }t\in \{0,1,\ldots,q-1\}.\]
The case $t\geq 3$ is clearly impossible.
If $t=1$, then $s\neq 2$ and so $s\geq 3$, and by computation we see that $q+1>2(1+q^{2/3})$ whenever $q\geq 11$, and $q+1\neq 2(1+q^{2/s})$ whenever $q\in \{5,7,9\}$ and $s\geq 3$.
Therefore $t=2$ and so $s=2$.
~\qed
\end{proof}

\section{Proof of Theorem~\ref{Thm-1}}

We will prove Theorem~\ref{Thm-1} in this section. First we treat the primitive case.

\begin{lemma}\label{PrimBiReg}
Let $G$ be a primitive permutation group  on a set $\Ome$ containing a bi-regular dihedral subgroup $H\cong \D_{2n}$. Then part $(1)$
of Theorem~$\ref{Thm-1}$ holds.
\end{lemma}
\proof
Write $H=\l a\r:\l b\r$ with $o(a)=n$, $o(b)=2$ and $a^b=a^{-1}$.
Since $H$ is bi-regular on $\Ome$, we see  that the permutation $a$ has cycle lengths $(n,n,n,n)$ on $\Ome$.
Hence $G$ satisfies one of parts (a)-(d) of Theorem~\ref{Less-4Cyc}.

Suppose $\soc(G)=\A_m^{\ell}$, as in part (a) of Theorem~\ref{Less-4Cyc}. From~\cite[Table~1]{GMPS16} we see that either
$\ell=1$ and $m=4n$ (Line~1 of Table~1 in \cite{GMPS16}), or $(m,\ell)=(8,2)$ and $n=16$ (Line~10 of~Table~1 in \cite{GMPS16}).
The former case gives rise to (1)(ii) of Theorem~$\ref{Thm-1}$.
For the latter case,  $G$ is primitive with socle $\A_8^2$,
so $G\le \S_8\wr\ZZ_2$, but a computation in Magma~\cite{Magma} shows that $\S_8\wr\ZZ_2$
has no bi-regular subgroups isomorphic to $\D_{32}$, a contradiction.

Suppose $\soc(G)=\PSL_d(q)^{\ell}$, as in part (b) of Theorem~\ref{Less-4Cyc}.
Checking \cite[Table~2]{GMPS16} we see that the possible candidates are as following:
\begin{itemize}
\item[(b.1)] $\ell=1$, $4\mid d$, $q \equiv 1 \pmod{4}$ and $n=(q^d-1)/(4(q-1))$  (Line~5 of~Table~2);
\item[(b.2)] $\ell=1$, $2 \mid d$, $q \equiv 3 \pmod{4}$ and $n=(q^d-1)/(4(q-1))$ (Line~6 of~Table~2);
\item[(b.3)] $\ell=2$, $d=3$, $q=2$ and $n=16$ (Line~30 of~Table~2).
\end{itemize}

Assume (b.1) or (b.2) occurs.
If $d\geq 4$, then $G\leq \Aut(\PSL_d(q))\cong\PGammaL_d(q)$,
and the action of $G$ on $\Ome$ is the natural action on the projective space $\PG_{d-1}(q)$.
If $a\in \PGammaL_d(q)\setminus\PGL_d(q)$, by \cite[Lemmas~4.2]{GMPS16},
$d,q$ and the cycle lengths of $a$ are as in Lines 10--17, 19--20, 23 of \cite[Table~2]{GMPS16}.
For Line 23, $a$ has cycle lengths $(10,15,30,30)$,
and for the other lines $d=2$, a contradiction.
Hence $a \in \PGL_d(q)$ lies in a Singer subgroup of $\PGL_d(q)$ of order $(q^d-1)/(q-1)$, refer to \cite[Lemma~4.4]{GMPS16}.
Notice that the order of $a$ is $(q^d-1)/4(q-1)$.
Let $\l x\r=\GL_1(q^d) $ be the Singer subgroup of $\GL_d(q)$ and let $Z=\langle x^{(q^d-1)/(q-1)}\rangle$ be the center of $\GL_d(q)$.
Identify $\PGammaL_d(q)$ with $\GammaL_d(q)/Z$.
Then we may identify  $a$ with $x^4Z$.
Write $q=p^e$ for some prime $p$ and positive integer $e$.
Notice that $\N_{\GammaL_d(q)}(\langle x^4 \rangle)=\N_{\GammaL_d(q)}(\GL_1(q^d))=\GL_1(q^d){:}\ZZ_{de}$.
Hence,  $b=gZ$ for some $g\in \GL_1(q^d){:}\ZZ_{de}$.
Since $a^b=a^{-1}$, we have $x^{4}(x^{4})^g\in Z$.
By Lemma~\ref{lm:mdq}, we  obtain that $d=4$ and $q=3$, and so $G=\PSL_4(3)$ or $\PGL_4(3)$.
However, computation in {\sc Magma}~\cite{Magma} shows that both $\PSL_4(3)$ and $\PGL_4(3)$
acting on the projective space $\PG_3(3)$ of degree $(3^4-1)/(3-1)=40$ have no bi-regular dihedral group $\D_{20}$,
a contradiction. Hence $d=2$ and (b.2) happens, which gives part (1)(ii) of Theorem~\ref{Thm-1}.

For (b.3), $\soc(G)=\PSL_3(2)^2$ and $\PSL_3(2)^2.\ZZ_2\le G\le \Aut(\PSL_3(2)^2)\cong (\PSL_3(2).\ZZ_2)\wr\ZZ_2$,
but computation in {\sc Magma}~\cite{Magma} shows that $G$ has no bi-regular subgroup $\D_{32}$, a contradiction.

Suppose that part (c) of Theorem~\ref{Less-4Cyc} occurs.
Checking the candidates in \cite[Tables~3--5]{GMPS16} we conclude that the possible pairs of $(\soc(G),n)$ are as following:
\[
\begin{split}
& (\A_5^2,15), (\A_8,7),  (\A_9,9),  (\mathrm{M}_{12},3),  (\mathrm{M}_{24},6), (\PSL_2(8),7), (\PSL_2(8),9),(\PSL_2(16),17),  \\
& (\PSU_3(3),7),  (\PSU_4(3),28), (\PSp_6(2),7), (\PSp_6(2),9), (\PSp_4(3),9),    (\PSp_4(3),10) .
\end{split}
\]
For each pair $(\soc(G),n)$ above,  
computation in Magma~\cite{Magma} shows that $G$ has a bi-regular  subgroup $\D_{2n}$ if and only if $(G,n)=(\mathrm{M}_{12},3)$ and $(\mathrm{M}_{24},6)$,
which give rise to part (1)(iv) and (1)(v) of Theorem~\ref{Thm-1}.

Finally suppose part (d) of Theorem~\ref{Less-4Cyc} occurs.
Now $\soc(G)=\ZZ_p^d$ with $p$ a prime is regular on $\Ome$, $n=p^d/4$,
and $G\le\AGL_d(p)$ is described in Theorems 1.3 and 1.4 of \cite{GMPS15}.
According to~\cite[Theorem 1.5]{GMPS15}, the permutation $g$ appears in~\cite[Tables 5--7]{GMPS15}.
From these three tables one can easily conclude that $(p,d)=(3,2)$, $(4,2)$ or $(5,2)$,
computation in {\sc Magma}~\cite{Magma} shows that $G$ has  a bi-regular  subgroup $\D_{2n}$ if and only if
$G$ satisfies one in part (1)(i) of Theorem~\ref{Thm-1}.\qed

To treat the imprimitive case, we need the following permutation group version of \cite[Lemma 3.1(1)]{DGJ22}.

\begin{lemma}\label{Block}
Let $G$ be a transitive permutation group on a set $\Ome$, and let $H\le G$ be a biregular subgroup with two orbits $H_0$ and $H_1$ on $\Ome$.
Let $\BB$ be a $G$-invariant partition of $\Ome$.
Then  either all elements of $\BB$ are subsets of $H_0$ or $H_1$, or $B\cap H_0\ne \emptyset$ and $B\cap H_1\ne\emptyset$.
\end{lemma}

\vskip0.1in
For a transitive permutation group $G$ on a set $\Ome$, a $G$-invariant partition $\BB$ of $\Ome$ is called
{\it maximal} if for any block $B\in \BB$ there exists no nontrivial $G$-block $C$ such that $B$ is a proper set $C$.
In this case, the induced group $G^{\BB}$ is primitive.

\begin{lemma}\label{QP-bireg}
Let $G$ be a quasiprimitive but imprimitive permutation group on a set $\Ome$
containing a biregular subgroup $H=\l a\r:\l b\r\cong\D_{2n}$.
Then part $(2)$ of Theorem~\ref{Thm-1} holds.
\end{lemma}

\proof Let $H_0$ and $H_1$ be the two orbits of $H$ on $\Ome$,
and let $\BB$ be a maximal $G$-invariant partition of $\Ome$.
Suppose that $B\in\BB$ contains a point $\a\in\Ome$.
Since $|G|/|G_\a|=|\Ome|=|B|\cdot|\BB|=|B|\cdot|G|/|G_B|$, $G_\alpha$ is of index $|B|$ in $G_{B}$.
Since $|\BB|\ge 2$ and $G$ is quasiprimitive on $\Ome$, we conclude that $G$ is faithful on $\BB$, and $G^{\BB}\cong G$ is primitive by the maximality of $\BB$.
Moreover, according to Lemma~\ref{Block}, either all elements of $\BB$ are subsets of $H_0$ or $H_1$, or $B\cap H_0\ne \emptyset$ and $B\cap H_1\ne\emptyset$.
We treat these two cases separately.

\vskip0.1in
{\noindent\bf Case 1.} Suppose that all elements of $\BB$ are subsets of $H_0$ or $H_1$.

In this case, $|B|$ divides $|H|$,
and we may assume $H_0=B_1\cup\cdots \cup B_t$ and $H_1=B_{t+1}\cup\cdots\cup B_{2t}$
such that $\BB=\{B_1,\dots,B_t,B_{t+1},\dots,B_{2t}\}$.
Since $H$ is regular on both $H_0$ and $H_1$,
$H$ acts transitively on $\{B_1,\dots,B_t\}$ and $\{B_{t+1},\dots,B_{2t}\}$,
and hence $|H|/|H_{B_1}|=\cdots=|H|/|H_{B_{2t}}|=|\BB|/2$.
If $|\BB|/2<n$,
then Lemma~\ref{Subgroup} implies that $H_{B_1}\cap\cdots \cap H_{B_{2t}}$ contains a nontrivial normal subgroup of $H$, which means that $H$ is not faithful on $\BB$, a contradiction.
Thus $|\BB|/2=n$,  and so $t=n$, $|B|=2$ and $|H_{B_i}|=2$.
Since $H$ is transitive on $\{B_1,\dots,B_t\}$ and on $\{B_{t+1},\dots,B_{2t}\}$,
$H_{B_1},\dots,H_{B_n}$ are conjugate in $H$,
and also $H_{B_{n+1}},\ldots,H_{B_{2n}}$ are conjugate in $H$.
Notice that $H_{B_1}\cap\cdots \cap H_{B_{2n}}$ contains no nontrivial normal subgroup of $H$.
Therefore, there are following two possibilities (interchanging $H_{B_1}$ and $H_{B_{n+1}}$ if necessary):
\begin{itemize}
\item[(a)] $n$ is even, $H_{B_1}=\l a^ib\r$ for some $i$ and $H_{B_{n+1}}=\l a^{n/2}\r$;
\item[(b)] $H_{B_1}=\l a^ib\r$ and $H_{B_{n+1}}=\l a^jb\r$  for some $i$ and $j$.
\end{itemize}

Assume that case (a) occurs. Since $\l a\r H_{B_1}=H$, $\l a\r\le G^{\BB}$ is transitive and so regular on $\{B_1,\dots,B_t\}$.
Since $H_{B_{n+1}}\le \l a\r$ and $H_{B_{n+1}},\ldots,H_{B_{2n}}$ are conjugate in $H$,
we have that $H_{B_{j}}=\l a^{n/2}\r$ for each $j\in\{n+1,\dots,2n\}$,
it follows that each orbit of $\l a\r$ on $\{B_{t+1},\dots,B_{2t}\}$ has length $|\l a\r:\l a^{n/2}\r|=n/2$.
Thereby $a$ has cycle lengths $(n,n/2,n/2)$ on $\BB$.
Applying Theorem~\ref{Less-4Cyc} to the primitive permutation group $G^{\BB}$ of degree $2n$
with an element $a$ having cycle lengths $(n,n/2,n/2)$, we may read out that
the only possibility is $\soc(G^{\BB})=\A_{2n}$ with $n\geq 3$ in its natural action of degree $2n$ (Line 1 of \cite[Table~1]{GMPS16}).
Now  $(G,G_B)=(\A_{2n},\A_{2n-1})$ or $(\S_{2n},\S_{2n-1})$.
Since  $G_B$ has a subgroup $G_\alpha $ with index $|B|=2$, we further have that  $(G,G_B)=(\S_{2n},\S_{2n-1})$ and $G_\alpha=\A_{2n-1}$.
Noticing that $G_\alpha=G_\a'\le G'=\soc(G)\cong\A_{2n}$,
we conclude that the orbit of $\soc(G)$ containing $\alpha$ has length $|\soc(G):\soc(G)_\a|=2n=|\Ome|/2$,
which means that $\soc(G)$ is not transitive on $\Ome$, contradicting the quasiprimitivity of $G$ on $\Ome$.

Assume that  case (b) occurs.
Then $\l a\r H_{B_1}=\l a\r H_{B_{n+1}}=H$,
so $\l a\r \le G^{\BB}$ is regular on both $\{B_1,\dots,B_n\}$ and $\{B_{t+1},\dots,B_{2n}\}$, and  hence $a$ has cycle lengths $(n,n)$ on $\BB$.
Therefore, $G^{\BB}\cong G$ is a primitive permutation group  containing a biregular cyclic subgroup $\l a\r\cong\ZZ_n$.
By the classification of such groups given in \cite[Theorem 3.3]{Muller13}, one of the following holds:
\begin{itemize}
\item[\rm (b.1)] $\soc(G^{\BB})=\A_{2n}$ (with $n\geq 3$) in its natural action of degree $2n$; 
\item[\rm (b.2)] $\soc(G^{\BB})=\PSL_d(q)$ in its natural action of degree $(q^d-1)/ (q-1) $, where $d$ is even and $q$ is odd, and $n=(q^d-1)/ (2(q-1))$; 
\item[\rm (b.3)] $(\soc(G^{\BB}),n)=(\A_5,5)$, 
$(\M_{12},6)$, 
$(\M_{22},11)$,  $(\M_{24},12)$, 
$(\ZZ_2^2,2)$, $(\ZZ_{2}^3,4)$, 
or $(\ZZ_2^4,8)$. 
\end{itemize}

For case (b.1),
as $ G_\alpha $ is of index 2 in $G_B$, we derive that $(G,G_B)=(\S_{2n},\S_{2n-1})$ and $G_\alpha=\A_{2n-1}$.
Then it follows from $G_\alpha \le\soc(G)=\A_{2n}$  that $\soc(G)$ is not transitive on $\Ome$,  a contradiction.

Next consider case (b.2). Then $G\leq \mathrm{P\Gamma L}_d(q)$.
From~\cite[Lemmas~4.2~and~4.4]{GMPS16} we see that   $  a $ lies in a Singer subgroup of $\PGL_d(q)$ of order $(q^d-1)/(q-1)$.
 Notice that the order of $a$ is $(q^d-1)/2(q-1)$.
Let $\l x\r=\GL_1(q^d) $ be the Singer subgroup of $\GL_d(q)$ and let $Z=\langle x^{(q^d-1)/(q-1)}\rangle$ be the center of $\GL_d(q)$.
Identify $\PGammaL_d(q)$ with $\GammaL_d(q)/Z$.
Then we may identify  $a$ with $x^2Z$.
Write $q=p^e$ for some prime $p$ and positive integer $e$, and write $\N_{\GammaL_d(q)}(\langle x^4 \rangle) =\langle x\rangle:\langle y\rangle $, where $\langle y\rangle \cong \ZZ_{de}$.
Now $b=gZ$ for some $g\in \GL_1(q^d):\ZZ_{de}$.
Since  $o(x^2)=(q^d-1)/2(q-1)$ is not a divisor of $q^{d'}-1$ for any $1\leq d' <d$, we have $ \N_{\GammaL_d(q)}(\langle x^2 \rangle)=\langle x\rangle:\langle y\rangle\cong \GammaL_1(p^{de})$ with $\langle y\rangle \cong \ZZ_{de}$ and $x^y=x^p$.
Then $b=gZ$ for some $g\in \langle x\rangle:\langle y\rangle $.
Since $a^b=a^{-1}$, we have $x^{2}(x^{2})^g\in Z$.
By Lemma~\ref{lm:mdq} (corresponding to the case $m=2$), we conclude that $d=2$ and $g=x^iy^{e}$ for some $0\leq i \leq q^{2}-1$.
It follows that $H$ is either $H_1:=\langle x^2Z, y^{e}Z \rangle/Z$ or $H_2:=\langle x^2Z, xy^{e}Z \rangle/Z$.
Let $M=\langle xZ,y^{e}Z\rangle$.
From~\cite[p.187, Satz~7.3]{Huppert1967} we see that $M \leq\GL_2(q)/Z$.
Then $H<M<\PGL_2(q)$.
Notice that $M \cong \D_{2(q+1)}$ as $xx^{y^{e}}=x^{ 1+p^e }=x^{1+q}\in Z$, and that $\PGL_2(q)$ has only one conjugate class of $\D_{2(q+1)}$, and every $\D_{2(q+1)} $ of $\PGL_2(q)$ has two  conjugate classes of subgroups $\D_{ q+1 }$, of which one is contained in $\PSL_2(q)$ and the other is not contained in $\PSL_2(q)$.
Clearly, $H_1\leq \PSL_2(q)$ while $H_2 \nleq \PSL_2(q)$.
It follows that $H$ is either $ \langle x^2Z, y^{e}Z \rangle $ or $ \langle x^2Z, xy^{e}Z \rangle $.
Let $M=\langle xZ,y^{e}Z\rangle$.
From $xx^{y^{e}}=x^{ 1+p^e }=x^{1+q}\in Z$ we see $M \cong \D_{2(q+1)}$.
Moreover, since $\mathbf{N}_{\mathrm{GL}_2(q)}(\langle x\rangle)\cong \langle x\rangle.2$ (see~\cite[p.187, Satz~7.3]{Huppert1967}), we obtain $M \leq\GL_2(q)/Z$.
Then $H<M<\PGL_2(q)$.

We claim that $G \cap \PGL_2(q)=\PSL_2(q)$.
Suppose for a contradiction that $G \cap \PGL_2(q)>\PSL_2(q)$.
Then $L:=\PGL_2(q) \leq G$, and since $L$ is normal in $\Aut(\PSL_2(q))$,  it is also normal in $G$.
By the transitivity of $L$ on both $\BB$ and $\Ome$, we conclude that $L_\alpha$ is of index $2$ in $L_{B}=\ZZ_p^e : \ZZ_{ p^e-1 } $ and so $L_\alpha=\ZZ_p^e : \ZZ_{ (p^e-1)/2 } < \PSL_2(q)$, which implies that $\PSL_2(q)$ is not transitive on $\Ome$, a contradiction.
Therefore, the claim holds, that is, $G \cap \PGL_2(q)=\PSL_2(q)$.
Since $G/(G \cap \PGL_2(q))\cong \PGL_2(q)G/\PGL_2(q) \leq  \ZZ_e$, we derive that $G=\PSL_2(q).f$ for some $f\mid e$.
Furthermore,  the claim implies that $H <\PSL_2(q)$ because if not then $G\geq \langle \PSL_2(q),H \rangle=\PGL_2(q)$.

Since $T:=\PSL_2(q)$ is transitive on both $\BB$ and $\Ome$, the group $T_\alpha $  is a subgroup with index $|B|=2$ of $ T_B=\ZZ_p^e : \ZZ_{ (p^e-1)/2 }$, and hence $T_\alpha=\ZZ_p^e : \ZZ_{ (p^e-1)/4 } $.
This implies that  $ p^e \equiv 1\pmod{4}$.
Note that $\PSL_2(p^e)$  has only one conjugate class of involutions (see  for example~\cite[Table~B.1]{BG-book}).
If $p^e \equiv 1\pmod{8}$, then both $H$ and $T_\alpha$ contain involutions of $T$, which implies that $H$ is not semiregular on $\Ome$, a contradiction.
Therefore, $p^e \equiv 5\pmod{8}$.
This proves part (ii.1) of the lemma.

Finally consider case (b.3). Recall that $G^{\BB}\cong G$.
If  $(\soc(G^{\BB}),n)= (\M_{12},6)$, then $G=\M_{12}$ or $\Aut(\M_{12})\cong\M_{12}.2$,
as $\Aut(\M_{12})$ has no subgroup with index 12 \cite{Atlas}, we conclude that
$(G,G_B)= (\M_{12},\M_{11})$,
which is impossible as $G_B$ has a subgroup $G_\a$ with index 2.
If  $(\soc(G^{\BB}),n)= (\M_{22},11)$, then $(G,G_B)= (\M_{22},\PSL_3(4))$ or $(\M_{22}.2,\PSL_3(4).2) $.
The former case contradicts that $G_B$ has a subgroup $G_\a$ with index 2,
and for the latter case we have $G_\alpha<\soc(G)\cong\PSL_3(4)$, which implies that $\soc(G)$  is not transitive on $\Ome$,
also a contradiction.
If  $(\soc(G^{\BB}),n)= (\M_{24},12)$,
then $(G,G_B)= (\M_{24},\M_{23})$ as $\Out(\M_{24})=1$,
which is a contradiction as $G_B$ has a subgroup $G_\a$ with index 2.
If $(\soc(G^{\BB}),n)=(\ZZ_2^2,2)$, $(\ZZ_{2}^3,4)$ or $ (\ZZ_2^4,8)$, then $|\soc(G)|=2n=|\Ome|/2$,
so $\soc(G)$ is intransitive on $\Ome$, contradicting the quasiprimitivity of $G$ on $\Ome$.
Therefore, $(\soc(G^{\BB}),n)=(\A_5,5)$.
Now $G=\A_5$ or $\S_5$, which gives rise to examples as in part(2)(i) of Theorem~\ref{Thm-1}.

\vskip0.1in
{\noindent\bf Case 2.} Suppose that  $B\cap H_0\ne \emptyset$ and $B\cap H_1\ne\emptyset$ for each $B\in\BB$.

Since $H$ is regular on both $H_0$ and $H_1$,
$H$ is transitive on $\BB$.
Recall that $G$ acts faithfully on $\BB$, so does $H$,
hence $H_B$ is core-free in $H$.
Since $H$ is a dihedral group, either
\begin{itemize}
\item[(a)] $H_B=1$; or
\item[(b)] $H_B=\l a^ib\r\cong\ZZ_2$ for some $i$.
\end{itemize}

First assume case (a) occurs. Then $H$ is regular on $\BB$, and so $|\BB|=2n$, $|B|=2$  and $G_\a$ is a subgroup of $G_B$ with index $2$.
Now $G^{\BB}\cong G$ is a primitive group of degree $2n$  containing a regular dihedral group $H^{\BB}\cong H$,
hence $(G^{\BB},G_B,H)$ (as $(G,G_w,H)$ there) satisfies \cite[Theorem 1.5]{Li-TAMS05}. Notice that $G_B$ has a subgroup of index 2,
the possible candidates are as following:
\begin{itemize}
\item [(a.1)] $(G^{\BB},G_B,n)=(\S_4,\S_3,2)$, $ (\M_{22}.2,\PSL_3(4).2 ,11)$  or $(\S_{2n},\S_{2n-1},n)$ with $n\geq 3$;
\item [(a.2)] $G^{\BB}=\PSL_2(p^e).o$, $H=\D_{p^e+1}$ and $G_B=\ZZ_p^e:\ZZ_{p^e-1\over2}.o$ where $p^e\equiv3\ (\mod 4)$ and $o\le\Out(\PSL_2(p^f))\cong\ZZ_2\times\ZZ_e$;
\item[(a.3)] $G^{\BB}=\PGL_2(p^e).\ZZ_f$, $H=\D_{p^e+1}$ and $G_B=(\ZZ_p^e:\ZZ_{p^e-1}).\ZZ_f$ where $p^e\equiv 1 ~(\mod 4)$ and $f\mid e$.
\end{itemize}

If $(G^{\BB},G_B,n)=(\S_4,\S_3,2)$, then $|\Ome|=8$ and so $\soc(G)=\ZZ_2^2$ is not transitive on $\Ome$, contradicting the quasiprimitivity of $G$.
If $(G^{\BB},G_B,n)=(\M_{22}.2,\PSL_3(4).2 ,11)$  or $(\S_{2n},\S_{2n-1},n)$, then
$G_\a=\PSL_3(4)$ or $\A_{2n-1}$ is contained in $\soc(G)$. It follows that
$$|\a^{\soc(G)}|=|\soc(G):\soc(G)_\a|=|\soc(G):G_\a|=2n=|\Ome|/2,$$
namely $\soc(G)$ is not transitive on $\Ome$, also contradicting the quasiprimitivity of $G$.

For case (a.2), since $\soc(G)=\PSL_2(p^e)$ is transitive on both $\BB$ and $\Ome$, we derive that
$T_B=\ZZ_p^e : \ZZ_{(p^e-1)/2}$ and $|T_B|/|T_\a|=2$.
However, as $p^e\equiv 3(\mod 4)$, $\ZZ_p^e : \ZZ_{(p^e-1)/2}$ has no subgroup with index $2$, a contradiction.

For case (a.3), let $N\cong\ZZ_p^e:\ZZ_{ p^e-1 }$ be the normal subgroup of $G_{B}=(\ZZ_p^e:\ZZ_{ p^e-1}).\ZZ_f$.
Since $G_\a$ is of index $2$ in $G_{B}$, we see that $N\cap G_\a\cong\ZZ_p^e:\ZZ_{p^e-1 }$ or $\ZZ_p^e:\ZZ_{(p^e-1)/2}$,
so $G_\alpha\cong(\ZZ_p^e:\ZZ_{ p^e-1 }).\ZZ_{f/2}$ with $f$ even or $G_\a\cong(\ZZ_p^e:\ZZ_{(p^e-1)/2}).\ZZ_f$,
and hence $G_\a$ is contained in $\PGL_2(q).\ZZ_{f/2}$ or $\PSL_2(q).\ZZ_f $.
Notice that both $\PGL_2(q).\ZZ_{f/2}$ and $\PSL_2(q).\ZZ_f $ are normal subgroups with index $2$ of $G$.
It follows that either $\PGL_2(q).\ZZ_{f/2}$ or $\PSL_2(q).\ZZ_f $ is intransitive on $\Ome$, a contradiction.

\vskip0.1in
Now we assume case (b) occurs, that is,  $H_B=\l a^ib\r$ for some $i$. Then $|\BB|=|H|/|H_B|=n$, and so $|B|=4$ and $G_\a$ is a subgroup of index $4$ in $G_B$.
Recall that $G$ is faithful on $\BB$ and $H$ is transitive on $\BB$. Since $\l a\r H_B=H$, we further conclude
that the group $\l a\r$ is transitive  and so regular on $\BB$.
Now $G^{\BB}\cong G$ is a primitive group of degree $n$  containing a regular cyclic group,
and $G_B$ has a subgroup with index 4,
according to~\cite{Jones02} or \cite[Corollary 1.2]{Li-Proc03}, one of the following holds:

\begin{itemize}
\item[(b.1)] $\ZZ_p\leq G^{\BB}\leq \AGL_1(p)$ and $n=p$;
\item[(b.2)] $\PGL_d(q)\leq G^{\BB}\leq \PGammaL_d(q)$ and $n=(q^d-1)/(q-1)$;
\end{itemize}

For case (b.1), we have $|\Ome|=4p$ and $\soc(G)\cong\ZZ_p$ is intransitive on $\Ome$, contradicts that
$G$ is quasiprimitive on $\Ome$.

Consider case (b.2).
Then $a \in G^{\BB}$ has cycle length $(n)$ in $\BB$, and by~\cite[Lemmas~4.2~and~4.4]{GMPS16} we see that   $a$ lies in a Singer cycle of $\PGL_d(q)$ of order $(q^d-1)/(q-1)$.
If $d\geq 3$, then by the same argument as in Lemma~\ref{PrimBiReg} and Lemma~\ref{lm:mdq} (for the case $m=1$), we conclude that $G$ contains no such bi-regular dihedral group $H$.
Therefore, $d=2$.
Write $q=p^e$ for some prime $p$ and positive integer $e$, and let $G=\PGL_2(q).f$ for some $f \mid e$.
  Recall that $G_\alpha$ is a subgroup with index $4$ of $G_B= \ZZ_p^e : (\ZZ_{ p^e-1}.f)$.
Let $N=\ZZ_p^e \unlhd G_B$.

Assume  $p=2$.
If $e=2$, then $\PGL_2(q)\cong \A_5$, and part (ii.2) holds.
We show next that the case $e\geq 3$ is impossible.
Suppose $e\geq 3$.
Then $G_\alpha \cap N \neq 1$.
Since $|G_B|/|G_\alpha|=4$, we have  $\ZZ_{ 2^e-1 } \unlhd G_\a/(G_\alpha \cap N)$.
Note that  $\ZZ_{ 2^e-1 }$ is irreducible on $\ZZ_2^e$.
It follows that $\ZZ_2^e : \ZZ_{ 2^e-1 } \unlhd G_\a$, and hence $G_\a=\ZZ_2^e : \ZZ_{ 2^e-1 }.f/4$.
So $G_\a$ is contained in the normal subgroup $\PGL_d(q).f/4$ of $G$, which implies that $\PGL_d(q).f/4$ is not transitive on $\Ome$, a contradiction.

Assume that $p$ is an odd prime.
Since $G_\alpha$ has index with $4$ in $G_B$, we derive that  $N=\ZZ_p^e <G_\alpha$, and  $G_\alpha=\ZZ_p^e : \ZZ_{ (p^e-1)/4 }.f $, or $\ZZ_p^e : \ZZ_{ (p^e-1)/2 }.f/2 $ or $ \ZZ_p^e : \ZZ_{ (p^e-1)}.f/4  $.
Note that both $\ZZ_p^e : \ZZ_{ (p^e-1)/4 }.f$ and  $\ZZ_p^e : \ZZ_{ (p^e-1)/2 }.f/2$ are subgroups of $ \PSL_2(q).f$, and $ \ZZ_p^e : \ZZ_{ (p^e-1)}.f/4 <\PGL_2(q).f/4$, and both $\PSL_2(q).f$ and $\PGL_2(q).f/4 $ are proper normal subgroups of $G=\PGL_2(q).f$.
Therefore, either $\PSL_2(q).f$ or $\PGL_2(q).f/4 $ is not transitive on $\Ome$, a contradiction.\qed

Combining  Lemma~\ref{PrimBiReg} and Lemma~\ref{QP-bireg}, we complete the proof of Theorem~\ref{Thm-1}.

\section{Examples of graphs}

We here introduce examples appearing in Theorem~\ref{Thm-2}.

\begin{example}\label{ex:affinepolargraph}
Let $m$ be a positive integer and let $V$  be a vector space of dimension $2m$ over $\mathbb{F}_q$,   equipped with a nondegenerate quadratic form $Q$ of type $\epsilon \in \{+,-\}$.
Following~\cite[\textsection{3.3.1}]{BrouwerSRG}, the \emph{affine polar graph} $\mathrm{VO}^{\epsilon}_{2m}(q)$ is a graph with vertex set  $V$, and  $v$ and $w$ being adjacent if and only if $Q(v-w)=0$.
For every $\epsilon \in \{+,- \}$, the valency of $\mathrm{VO}^{\epsilon}_{2m}(q)$ is $(q^m-\epsilon)(q^{m-1}+\epsilon)$.
It can be verified in~{\sc Magma}~\cite{Magma} that both $\mathrm{VO}^{+}_{4}(2)$ and $\mathrm{VO}^{-}_{4}(2)$ are connected arc-transitive bi-Cayley graphs of $\D_{8}$, and their full automorphism groups are $\ZZ_2^4 : \mathrm{SO}^{+}_{4}(2)$ and $\ZZ_2^4 : \mathrm{SO}^{-}_{4}(2)$, respectively.
Notice that, for every $\epsilon \in \{+,- \}$, the group  $\ZZ_2^4 : \mathrm{SO}^{\epsilon}_{4}(2)$ is primitive and has rank $3$ on $V$.
Furthermore, the graph $\mathrm{VO}^{+}_{4}(2)$ is isomorphic to the \emph{Hamming graph} $\mathrm{H}(2,4)$ (whose vertex set  is $\{1,2,3,4 \}^2$, and two distinct vertices $(i,j)$ and $(i',j')$ being adjacent if and only if $i=i'$ or $j=j'$).
\end{example}

\begin{example}
Let $q$ be a prime power and let $d\geqslant 3$ be an integer. Let
$\mathrm{PG}(d-1,q)$ be the $(d-1)$-dimensional projective geometry over a
field of order $q$. Let $V_1$ and $V_{n-1}$ be the sets of points
and hyperplanes of $\mathrm{PG}(d-1,q)$, respectively. The {\em
point-hyperplane incidence graph}, denoted by $B(\mathrm{PG}(d-1,q))$, of
$\mathrm{PG}(d-1,q)$ is the graph with vertices $V_1\cup V_{n-1}$, and
$\a\in V_1$ and $\b\in V_{n-1}$ are adjacent if and only if $\a$ is
a point of $\b$ (see, for example,  \cite{DGJ22}).
Let $B'(\mathrm{PG}(d-1,q))$ be the bipartite complement of $B(\mathrm{PG}(d-1,q))$,  that is, the bipartite graph with the same vertex set but a $1$-subspace and a $(d-1)$-subspace are adjacent if and only if their intersection is the zero subspace.
Notice that $\Aut(B(\mathrm{PG}(d-1,q)))\cong \Aut(B'(\mathrm{PG}(d-1,q)))\cong \Aut(\PSL_{d}(q))=\PGammaL_d(q).2$, where $2$ is an inverse-transpose automorphism (following~\cite[p.14]{K-Lie}).
\end{example}

The following lemma was pointed out in \cite{DMM08} without a proof. We here present a simple proof.

\begin{lemma}\label{lm:BPGd-1q}
Both $B(\mathrm{PG}(d-1,q))$ and $B'(\mathrm{PG}(d-1,q))$ are  Cayley graphs on $\D_{2(q^d-1)/(q-1)}$,
and hence are bi-Cayley graphs on $\D_{(q^d-1)/(q-1)}$ if $(q^d-1)/(q-1)$ is even.
\end{lemma}

\begin{proof}
Fix a basis of $\mathbb{F}_q^d$ and let $\langle x \rangle\cong \ZZ_{q^d-1}$ be a Singer cycle of $\GL_d(q)$.
According to~\cite{TZ1959}, $x^{y_x}=x^{\mathsf{T}}$ for some symmetric matrix $y_x\in  \GL_d(q)$.
Let $\iota$ be the  inverse transpose map with respect to
the basis.
Then
\[(y_x\iota)^2=y_x\iota y_x\iota=y_x(y_x)^{\iota}=y_x  (y_x)^{-1\mathsf{T}} =y_x(y_x)^{-1}=1, \text{ and }  x^{y_x\iota}=(x^{\mathsf{T}})^{\iota}=x^{-1}. \]
This shows that $Y:=\langle x, y_x\iota\rangle =\langle x   \rangle : \langle   y_x\iota\rangle \cong \D_{2(q^d-1)}$.

Let $X=\GammaL_d(q) : \langle \iota\rangle$.
Then $X/Z\cong \Aut(\PSL_d(q))$ and $Y/Z\cong \D_{2(q^d-1)/(q-1)}$.
Recall the definitions of $B(\mathrm{PG}(d-1,q))$ and $B'(\mathrm{PG}(d-1,q))$.
Notice that $\langle x \rangle/Z$ acts regularly on both the set of $1$-dimensional subspaces of $\mathbb{F}_{q^d}$ and the set of $(d-1)$-dimensional subspaces of $\mathbb{F}_{q^d}$.
For every $y \in Y/\langle x\rangle  $, since $yZ/Z \notin \PGammaL_d(q)$, we see that $yZ/Z$ swaps the  set of $1$-dimensional subspaces and the set of  $(d-1)$-dimensional subspaces.
Hence $Y/Z$ acts regularly on the vertex sets of $B(\mathrm{PG}(d-1,q))$ and  $B'(\mathrm{PG}(d-1,q))$.\qed
\end{proof}

The coset graph is an important tool to construct and study arc-transitive graphs.
Let $G$ be a group, let $K$ be a core-free subgroup of $G$ and let $g\in G\setminus K$ such that $g^{-1}\in KgK$.
The \emph{coset graph} $\Cos(G,K,KgK)$ is  a graph with vertex set  the set of right cosets of $K$ in $X$, and two vertices $Kx$ is adjacent to $Ky$ with $x,y\in G$ if and only if $yx^{-1}\in KgK$.
The graph $\Cos(G,K,KgK)$ has order $|G|/|K|$ and valency $|K|/|K\cap K^g|$, and it is connected if and only of $\langle K,g \rangle=G$.
The next lemma is from \cite{Sab64}.

\begin{lemma}\label{Coset-gp}
\begin{itemize}
\item[(a)] Let $G$ be a group, let $K$ a core-free subgroup of $G$,  and let $g \in G$ satisfy  $g^{-1}\in KgK$ and $\langle K,g \rangle=G$.
Then $\Cos(G,K,KgK)$ is a connected $G$-arc-transitive graph with order $|G|/|K|$ and valence $|K|/|(K \cap K^g)|$.
Moreover, $\Cos(G,K,KgK )\cong \Cos(G,K^{\sigma},K^{\sigma} g ^{\sigma} K^{\sigma} )$ for every $\sigma\in \Aut(G)$.
\item[(b)] Let $\Ga $ be a connected $G$-arc-transitive graph with valence $d$ and let $\{ \alpha,\beta \}$ be an edge of $\G$.
Then $\Ga  \cong \Cos(G,G_\alpha,G_\alpha g G_\alpha )$ for every $g$ which swaps $\alpha$ and $\beta$.
\end{itemize}
 \end{lemma}

\begin{example}\label{ex:A5Z3}
Let $G=\A_5=\langle (1,2,3),(3, 4, 5) \rangle$, and let $K=\langle (1,2,3)\rangle\cong \ZZ_3$, and let $H =\langle (1,2,3,4,5),(2,5)(3,4)\rangle\cong \D_{10}$, and let $\Ome$ be the set of right $K$-cosets in $G$.
Consider the right multiplication of $G$ on $\Ome$.
Since $K$ is a $3$-group and $H$ is a $\{2,5\}$-group, it follows that $H\cap K^g=1$ for every $g \in G$, which implies that $H$ acts semiregularly on $\Omega$.
Notice that $|\Omega|=20$.
Therefore, $H$ is a bi-regular subgroup of $G$ on $\Omega$.
Computation in {\sc Magma}~\cite{Magma} shows that among the $G$-orbital graphs on $\Ome$, there is only one which is connected and $G$-arc-transitive.
Moreover, this graph is of valency $3$ and has full automorphism group $\ZZ_2 \times \A_5$, and it is the graph $\mathrm{F020A}$ in~Conder's list~\cite{Conder768}.

\end{example}

\begin{example}\label{ex:S5Z6S3}
Let $G=\S_5=\langle (1,2,3,4,5),(1,2) \rangle$, and let $K=\langle (1,2,3),(4,5)\rangle\cong \ZZ_6$, and let $H =\langle (1,2,3,4,5),(2,5)(3,4)\rangle\cong \D_{10}$, and let $\Ome$ be the set of right $K$-cosets in $G$.
Clearly,  $H\cap K^g=1$ or $\ZZ_2$ for every $g \in G$.
Since $K^g$ has only one involution and this involution is an odd permutation, we conclude from $K<\A_5$ that  $H\cap K^g=1$ for every $g \in G$.
Therefore, $H$ is a bi-regular subgroup of $G$ on $\Omega$.
Computation in {\sc Magma}~\cite{Magma} shows that  among the $G$-orbital graphs on $\Ome$, there are, up to isomorphism, three graphs which is connected and $G$-arc-transitive, namely $\GG_{20}^{(1)}$, $\GG_{20}^{(2)}$ and $\GG_{20}^{(3)}$.
Moreover, these three graphs have the same valency $6$, and $\Aut(\GG_{20}^{(1)})\cong \Aut(\GG_{20}^{(2)}) \cong \ZZ_2 \times \S_5$ and $\Aut(\GG_{20}^{(3)})\cong  \ZZ_2^{10} :  \S_5$.
\end{example}

\begin{example}\label{ex:2q}
Let $V$ be a vector space of dimension $2$ over $\mathbb{F}_q$, where  $q=p^e\equiv 5\pmod{8} $ with an odd prime $p$.
Fix a basis $(v_1,v_2)$ of $V$.
Let $\mu$ be a generator of $\mathbb{F}_q^{*} $.
Take  $Z=\langle \mathrm{diag}(\mu ,\mu )  \rangle$ to be the center of $ \GammaL_2(q) $ and let  $P=\langle \mathrm{diag}(\mu^2,\mu^2)  \rangle\cong \ZZ_{(q-1)/2}$.
Let  $X =\langle Z,\mathrm{\Sigma L}_2(q)\rangle $ be a subgroup of $\GammaL_2(q) $, where $\mathrm{\Sigma L}_2(q)=\SL_2(q) : \langle \phi\rangle$ with $\phi$ acting on $V$ by $ (a_1v_1+a_2v_2)^\phi= a_1^p v_1+a_2^p v_2$ for all $a_1,a_2\in \mathbb{F}_q$.
Let $\Ome$ be the set of orbits of $P$ on $V\setminus \{0\}$ $($so $|\Ome|=2(q+1))$ and consider the action of $X$ on $\Ome$  by $
(v^P)^x=(v^x)^{P}$,   for all $x \in X$, $v \in V\setminus \{ 0\}$ $($one may notice that this action is well defined$)$.
Clearly,  this action is with kernel $P$, and $X/P\cong \ZZ_2\times \PSigmaL_2(q)$.
Let $G=\PSL_2(q)$ be the subgroup of $X/P$, and let $K$ be the stabilizer of $v_1^P$ in $G$.
Notice that $K\cong \mathbb{Z}_p^e:\mathbb{Z}_{(p^e-1)/4}$ is the unique subgroup of $G_{\langle v_1\rangle}\cong \mathbb{Z}_p^e:\mathbb{Z}_{(p^e-1)/2}$ with index $2$.
Let $g \in \SL_2(q)$ be the linear transformation such that  $v_1^{ g}=v_2$ and $v_2^{ g}=-v_1$, and let $\overline{g}$ be $gP/P \in G$.
Define
\[
\mathcal{G}_{(2,q)}=\Cos(G,K,K \overline{g} K).
\]
Notice that the intersection of a Single subgroup of $\PGL_2(q)$ and $\PSL_2(q)$ is a cyclic group of order $(q+1)/2$, and the normalizer of this cyclic group in $\PSL_2(q)$ is a dihedral group $\D_{q+1}$.
Let $H$ be such a dihedral group $\D_{q+1}$.
\end{example}

\begin{lemma}\label{lm:cG2q}
Use notations as above. Then $\PSigmaL_2(q)$ acts quasiprimitively on $\Ome$ and $H$ acts bi-regularly on $\Ome$,
and $\mathcal{G}_{(2,q)}$ is the unique (up to isomorphism) $G$-orbital graphs on $\Ome$ that is connected arc-transitive,
and $\mathcal{G}_{(2,q)}$ is $G$-arc-transitive and of valency $q$, and  $\Aut(\mathcal{G}_{(2,q)})\geq \ZZ_2\times \PSigmaL_2(q) $.
\end{lemma}

\begin{proof}
We do computation in $X=\langle Z,\mathrm{\Sigma L}_2(q)\rangle $.
The stabilizer of $v_1^P $ in $\mathrm{\Sigma L}_2(q)$ is
$ (N  : M)  :  \langle  \phi \rangle$, where
\[
N=\left\{
 \begin{pmatrix}
1 & 0 \\
s & 1 \\
\end{pmatrix}  : s \in \mathbb{F}_{q} \right\} \cong \ZZ_p^e
,
M=\left\{
 \begin{pmatrix}
\lambda & 0 \\
0 & \lambda^{-1}\\
\end{pmatrix} :\lambda \in \langle \mu^2\rangle
\right\} \cong \ZZ_{(p^e-1)/2}.
\]
Since $|\mathrm{\Sigma L}_2(q)|/|\mathrm{\Sigma L}_2(q)_{v_1^P }|=2(q+1)$ and $|\Ome|=2(q+1)$, we see that $ \mathrm{\Sigma L}_2(q)$ acts transitively on $\Ome$.
Notice that  $v_1^P $  is also fixed by $y:=\mathrm{diag}\{u^2,1 \}=\mathrm{diag}\{u,u \}\cdot \mathrm{diag}\{u,u^{-1} \}$, where $\mathrm{diag}\{u,u \} \in Z$ and $\mathrm{diag}\{u,u^{-1}\} \in \SL_2(q)$.
Therefore, $X_{v_1^P} =\mathrm{\Sigma L}_2(q)_{v_1^P} :  \langle y\rangle$.
Since $\SL_2(q)_{v_1^P}=N : M$, we conclude that $G=\SL_2(q)$ also acts transitively on $\Ome$, and hence $\PSigmaL_2(q)$ is quasiprimitive  on $\Ome$.
Now, $K=G_{v_1^P}=\ZZ_{p}^e : \ZZ_{(q-1)/4}$, and so its order is coprime to $|H|= q+1 $, since $q\equiv 5\pmod{8}$.
This implies that  the group $H$ acts bi-regularly on $\Ome $.

It is clear that $ \mathrm{\Sigma L}_2(q)_{v_1^P}$ fixes $v_1^P$ and $(\mu v_1)^P$.
Note that $\Ome \setminus \{v_1^P,(\mu v_1)^P\}=\Ome_1\cup \Ome_2$, where
\[ \Ome_1=\{ (av_1+v_2)^P \rangle\mid a \in \mathbb{F}_{p^e} \},\Ome_2=\{ (\mu av_1+\mu v_2)^P \mid a \in \mathbb{F}_{p^e} \}. \]
Since $N$ is regular on both $\Ome_1$ and $\Ome_2$, and $M$ fixes both $\Ome_1$ and $\Ome_2$ setwise, we conclude that $\Ome_1$ and $\Ome_2$ are two orbits of $\mathrm{SL}_2(q)_{v_1^P}$ on $\Ome$.
Moreover, it is a straightforward verification to show that both $  \phi $ and $y$ fix  $\Ome_1$ and $\Ome_2$ setwise.
Therefor, $\SL_{2}(q)_{v_1^P}$ and $X_{v_1^P} $ has the same orbits on $\Ome$, which are $ v_1^P $, $(\mu v_1)^P$, $\Ome_1$ and $\Ome_2$.

Let $\mathit{\Sigma}_1$ and $\mathit{\Sigma}_2$ be the $\SL_2(q)$-orbital (or $X$-orbital) graphs corresponding to $\Ome_1$ and $\Ome_2$, respectively.
Recall that $g \in \SL_2(q)$ satisfies  $v_1^{ g}=v_2$ and $v_2^{ g}=-v_1$ and $\overline{g}=gP/P \in G $.
Let $h \in \SL_2(q)$ be  such that  $v_1^{ h}=\mu v_2$ and $v_2^{ g}=-\mu^{-1}v_1$, and let   $\overline{h}=hP/P \in G $.
Notice that  $-1 \in P$, since $ p^e \equiv 5\pmod{8}$.
Then $g$ swaps $v_1^P$ and $v_2^P$, and $h$ swaps $v_1^P$ and $(\mu v_2)^P$.
Since $v_1^{g^2}=-v_1$ and $v_2^{g^2}=-v_2$, and $v_1^{h^2}=-v_1$ and $v_2^{h^2}=-v_2$, we see that both $g^2$ and $h^2$ are in $Z$.
Therefore, the $\SL_2(q)$-orbitals (or $X$-orbitals) corresponding to $\Ome_1$ and $\Ome_2$ are self-paired, and hence  both $\mathit{\Sigma}_1$ and $\mathit{\Sigma}_2$ are $G$-arc-transitive (and $X/P$-arc-transitive) graphs of valency $p^e$ (notice that $|\Ome_1  |=|\Ome_2|=p^e$).
By Lemma~\ref{Coset-gp}, $\mathit{\Sigma}_1=\mathcal{G}_{(2,q)}=\Cos(G,K,K\overline{g} K)$, and  $\mathit{\Sigma}_2=\Cos(G,K,K \overline{h} K)$.

Let $\delta \in \GL_2(q)$ be such that $v_1^\delta=u^{-1}v_1$ and  $v_2^\delta= v_2$.
Direct verification shows that $(gP)^{\delta}= hP $ and $\sigma $ normalizes $X_{v_1^P}$.
It follows from Lemma~\ref{Coset-gp} that $\sigma$ induces an isomorphism between $\mathit{\Sigma}_1$ and $\mathit{\Sigma}_2$.
Moreover, by the classification of maximal subgroups of $\mathrm{PSL}_2(q)$ we conclude that $\langle G_{v_1^P},  \overline{g}\rangle=G$, and then $\langle (X/P)_{v_1^P}, \overline{g}\rangle=G$.  \qed
\end{proof}

\begin{example}\label{ex:d>3q}
Let $V$ be a vector space of dimension $d$ over $\mathbb{F}_q$, where $d\geq 3$ is an odd integer and  $q=p^e $ is an odd prime power.
Fix a basis $(v_1,v_2,\ldots,v_d)$ of $V$ and let $\iota$ be the inverse-transpose map with respect to this basis.
Let $X=\GammaL_d(q) : \langle \iota\rangle $ and let $X^{+}=\GammaL_d(q)=\GL_d(q):\langle \phi\rangle$, where  $\phi$ is defined as in Example~\ref{ex:2q}.
Let $\mu$ be a generator of $\mathbb{F}_q^{*} $ and let  $P=\langle \mathrm{diag}(\mu^2,\mu^2,\ldots,\mu^2)  \rangle\cong \ZZ_{(q-1)/2}$ be the subgroup with index $2$ of the center of $X^+$.
Let $\Ome $ be the set of orbits of $P$ on $V\setminus \{0\}$ $($so $|\Ome |=2(q^d-1)/(q-1))$ and consider the action of $X^{+}$ on $\Ome$  by $
(v^P)^x=(v^x)^{P}$ for all $x \in X^{+}$ and $v \in V\setminus \{ 0\}$.
This action is with kernel $P$, and $X^{+}/P\cong \ZZ_2\times \PGammaL_d(q)$, and $X /P\cong \ZZ_2\times (\PGammaL_d(q):\langle \tau\rangle) $, where $\PGammaL_d(q):\langle \tau\rangle=\Aut(\PSL_d(q))$.

Let $G=\GL_d(q)/P \cong \ZZ_2 \times \PGL_d(q) $ be the subgroup of $X^{+}/P$, and let $K$ be $(G_{v_1^P})^{\iota}$ and let $\Del$ be the set of right $K$-subsets in $G$ $($so $|\Del|=|\Ome|)$.
We shall show in Lemma~\ref{lm:cGdq} that $K$ has three orbits on $\Ome$, namely, $\Ome_i$ for all $i\in \{1,2,3\}$, and $|\Ome_1|=2(q^{d-1}-1)/(q-1)$, and $|\Ome_2|=|\Ome_3|=q^{d-1}$.
For every $\Ome_i$ with $i\in \{1,2,3\}$, define a graph $\mathcal{G}_{(d,q)}^{(i)}$ to be a bipartite graph with two parts $\Del$  and $\Ome$ as vertex set, and with $ \{K,\alpha\}^G$ for some $\alpha \in \Ome$ as edge set (notice the edge set is independent to the choice of $\alpha$ in $\Ome$).
Then every $\Ome_i$   is $G$-edge-transitive.
Define the action of $\iota$ on $\Ome \cap \Del$ by, for every $\alpha \in \Ome$ and $\beta \in \Del$, $\beta^\iota=\alpha$ if and only if $G_{\beta}^{\iota}=G_{\alpha}$.
Then  every $\Ome_i$   is  $G:\langle \tau\rangle $-arc-transitive.
Notice that $G_{v_1^P}$ is a subgroup with index $2$ of $ G_{v_1}$.
By Lemma~\ref{lm:BPGd-1q}, we conclude that every $\mathcal{G}_{(d,q)}^{(i)}$   admits a bi-regular $\mathrm{D}_{2(q^d-1)/(q-1)}$ that is contained in $G:\langle \tau\rangle $.
\end{example}

\begin{lemma}\label{lm:cGdq}
The following statements hold.
\begin{itemize}
\item[\rm (i)] $\mathcal{G}_{(d,q)}^{(2)}\cong \mathcal{G}_{(d,q)}^{(3)}$;
\item[\rm (ii)] for every $i\in \{1,2\}$, the graph $\mathcal{G}_{(d,q)}^{(i)}$ is connected, and $\Aut(\mathcal{G}_{(d,q)}^{(i)})\geq X/P\cong \ZZ_2\times \Aut(\PSL_d(q)) $;
\item[\rm (iii)] the group $\PSL_d(q)$ acts transitively on both $\Del$ and $\Ome$, and $K\cap \PSL_d(q)$ has the same orbits on $\Ome$ as $K$.
\end{itemize}
\end{lemma}

\begin{proof}
Recall that $G=\GL_d(q)/P$).
Notice that $\GL_d(q)_{v_1^P}= N:(M\times R) $, where
\begin{align*}
N=&\left\{
 \begin{pmatrix}
1 & \textbf{0} \\
\textbf{a}^{\mathsf{T}} & I_{n-1}\\
\end{pmatrix}  : I_{n-1}=\mathrm{diag}(1,\ldots,1) \in \GL_{d-1}(q),\ \textbf{a}  \in \mathbb{F}_{q}^{d-1} \right\} \cong  [q^{d-1}],\\
M=&\left\{
 \begin{pmatrix}
1 & \textbf{0} \\
\textbf{0}^{\mathsf{T}} & \textbf{b}\\
\end{pmatrix}  : \textbf{b} \in \GL_{d-1}(q) \right\} \cong \GL_{d-1}(q),\\
R=&\left\{
 \begin{pmatrix}
\lambda & \textbf{0} \\
\textbf{0}^{\mathsf{T}} & I_{n-1} \\
\end{pmatrix}    :  I_{n-1}=\mathrm{diag}(1,\ldots,1) \in \GL_{d-1}(q), \lambda \in \langle \mu^2 \rangle \right\} \cong \ZZ_{(q-1)/2} .
\end{align*}
It follows that $(\GL_d(q)_{v_1^P})^{\iota}= N_1:(M\times R) $, where
\begin{align*}
N_1=&\left\{
 \begin{pmatrix}
1 & \textbf{a} \\
\textbf{0}^{\mathsf{T}} & I_{n-1}\\
\end{pmatrix}  : I_{n-1}=\mathrm{diag}(1,\ldots,1) \in \GL_{d-1}(q),\ \textbf{a} \in \mathbb{F}_{q}^{d-1} \right\} \cong [q^{d-1}].
\end{align*}
Consider the orbits of $ (\GL_d(q)_{v_1^P})^{\iota} $ on $\Ome$.
Since $ (\GL_d(q)_{v_1^P})^{\iota} $ stabilizes the subspace $\langle v_2,\ldots,v_d\rangle$ and $M$ is transitive on all non-zero vectors of $\langle v_2,\ldots,v_d\rangle$, we see that $(\GL_d(q)_{v_1^P})^{\iota} $ has an orbit
\begin{align*}
\Ome_1: =\{ w^P: w \in  \langle v_2,\ldots,v_d\rangle\}.
\end{align*}
Clearly, $|\Ome_1|=2(q^{d-1}-1)/(q-1)$.
Notice that $N_1$ is transitive on all vectors in $\{ v_1+w: w \in w \in  \langle v_2,\ldots,v_d\rangle \}$.
Combine with the action of $R$, we see that $ (\GL_d(q)_{v_1^P})^{\iota} $ has two orbits on $V\setminus \langle  v_2,\ldots,v_d\rangle$, which are $\{ \lambda v_1+w: w  \in  \langle v_2,\ldots,v_d\rangle,\ \lambda \in \langle \mu^2\rangle \}$ and $\{ \lambda v_1+w: w  \in  \langle v_2,\ldots,v_d\rangle,\ \lambda \in  \langle \mu \rangle \setminus \langle \mu^2\rangle \}$.
Therefore, $ (\GL_d(q)_{v_1^P})^{\iota} $ has two orbits on $\Ome \setminus \Ome_1$, which are
\begin{align*}
\Ome_2:=&\{  (v_1+w)^P: w  \in  \langle v_2,\ldots,v_d\rangle \}, \\
\Ome_3:=&\{ (\mu v_1+w)^P: w  \in  \langle v_2,\ldots,v_d\rangle \}.
\end{align*}
Clearly, $|\Ome_2|=|\Ome_3|=q^{d-1}$.

Now we show $\mathcal{G}_{(d,q)}^{(2)}\cong \mathcal{G}_{(d,q)}^{(3)}$.
Consider the map $ \rho $ that fixes every point in $\Del$, and acts on $\Ome $ by $(v^P)^{\rho}=(\mu v)^P$ for every $v^P \in \Ome$.
Then $\Ome_2^{\rho}=\Ome_3^{\rho}$, and
\[ (v^P)^{\rho g}=((\mu v)^P)^g= (\mu v^g)^P =(v^P)^{g\rho } \text{ for every } g \in \GL_d(q). \]
It follows that
\[
\{K,v_1^P \}^{G\rho}=\{K, v_1^P \}^{\rho G}=\{K,(\mu v_1)^P \}^G,
\]
where $K \in \Del$ and $v_1^P , (\mu v_1)^P \in \Ome$.
Recall that the edge set of $\mathcal{G}_{(d,q)}^{(2)} $ is $\{K, v_1^P \}^G$ and that of  $\mathcal{G}_{(d,q)}^{(3)} $ is $\{K, (\mu v_1)^P \}^G$.
Therefore, $\rho$ induces an isomorphism from $\mathcal{G}_{(d,q)}^{(2)} $ to $\mathcal{G}_{(d,q)}^{(3)} $.

Next, we prove that both $\mathcal{G}_{(d,q)}^{(1)}$ and $  \mathcal{G}_{(d,q)}^{(2)}$ arc connected.
Let $g \in \SL_d(q)$  be such that
\begin{align*}
 &v_1^g=v_2,\ v_2^g=v_1,\ v_3^g=-v_3,\  w^g=w \text{ for all } w \in \{v_1,v_2,\ldots,v_d\}\setminus \{v_1,v_2,v_3\}.
\end{align*}
Then $((\GL_d(q)_{v_1^P})^{\iota} )^{\iota g}=\GL_d(q)_{v_2^P}$  and $((\GL_d(q)_{v_1^P})^{\iota} )^{\iota}=\GL_d(q)_{v_1^P}$.
By Lemma~\ref{Coset-gp},  it is suffices to show $\langle  (\GL_d(q)_{v_1^P})^{\iota}, \iota g\rangle=\GL_d(q)=\langle  (\GL_d(q)_{v_1^P})^{\iota}, \iota \rangle$.
Notice that $(\GL_d(q)_{v_1^P})^{\iota}$ is a subgroup of index $2$ of $\GL_d(q)_{\langle v_2,\ldots,v_d\rangle}$ and $\GL_d(q)_{\langle v_2,\ldots,v_d\rangle}$ is a maximal parabolic subgroup of $\GL_d(q)$.
Since both $ \GL_d(q)_{v_1^P} = (\GL_d(q)_{v_1^P})^{\iota^2}$ and $  (\GL_d(q)_{v_1^P})^g = (\GL_d(q)_{v_1^P})^{\iota^2g}$ do not stabilize $\langle v_2,\ldots,v_d\rangle $, we conclude that $\langle  (\GL_d(q)_{v_1^P})^{\iota}, \iota g\rangle=\GL_d(q)=\langle  (\GL_d(q)_{v_1^P})^{\iota}, \iota \rangle$, as required.

 Recall  that $\GammaL_d(q)=\GL_d(q):\langle \phi\rangle$.
It is a straightforward verification that $\phi $ stabilizes $v_1^P$, and normalizes $\GL_d(q)_{v_1^P}$ and commutes with both $\iota$ and $g$.
This implies that $\GammaL_d(q)$ is also transitive on both $\Del$ and $\Ome$, and $(\GammaL_d(q)_{v_1^P})^{\tau}$ has the same orbits on $\Ome$ as $(\GL_d(q)_{v_1^P})^{\tau}$.
Therefore, $\Aut(\mathcal{G}_{(d,q)}^{(1)})\geq X/P=\ZZ_2\times \Aut(\PSL_d(q))$.

Now we show that the group $\PSL_d(q)$ in $G=\ZZ_2\times \PGL_2(q)$ is transitive  on both $\Ome$  and $\Del$, and the orbits of $(\PSL_d(q)_{v_1^P})^{\iota}$ are same as that of $(\GL_d(q)_{v_1^P})^{\tau}$.
Notice that the center of $\SL_d(q) $, namely $Z_1$, is generated by $\mathrm{diag}\{ \mu^{k}, \ldots,\mu^{k}\}$, where $k=(q-1)/(d,q-1) $.
Since $d$ is odd, we have $Z_1\leq P$, which implies that $\SL_d(q) /(P \cap \SL_d(q)) =\SL_d(q)/Z_1\cong \PSL_d(q)$.
Therefore, we only need to show the assertion is true for $\SL_d(q)$.
Notice that $\SL_d(q)_{v_1^P}$ is $N:(M_1:R_1)$, where
\begin{align*}
N=&\left\{
 \begin{pmatrix}
1 & \textbf{0} \\
\textbf{a}^{\mathsf{T}} & I_{n-1}\\
\end{pmatrix}  : I_{n-1}=\mathrm{diag}(1,\ldots,1) \in \GL_{d-1}(q),\ \textbf{a}  \in \mathbb{F}_{q}^{d-1} \right\} \cong [q^{d-1}],\\
M_1=&\left\{
 \begin{pmatrix}
1 & \textbf{0} \\
\textbf{0}^{\mathsf{T}} & \textbf{b}\\
\end{pmatrix}  : \textbf{b} \in \SL_{d-1}(q) \right\} \cong \SL_{d-1}(q),\\
R_1=&\left\{
 \begin{pmatrix}
\lambda & \textbf{0} \\
\textbf{0}^{\mathsf{T}} & \textbf{c} \\
\end{pmatrix}    :  \textbf{c}=\mathrm{diag}(\lambda^{-1},1, \ldots,1) \in \GL_{d-1}(q), \lambda \in \langle \mu^2 \rangle  \right\}\cong \ZZ_{(q-1)/2}.
\end{align*}
Since $|\GL_d(q)|/|\SL_d(q)|=q-1$ and $|\GL_d(q)_{v_1^P}|/|\SL_d(q)_{v_1^P}|=q-1$, it follows that
$\SL_d(q)$ is transitive on both $\Ome$  and $\Del$.
Notice that $(\SL_d(q)_{v_1^P})^{\iota}$ is $N_1:(M_1:R_1)$.
Since $(\SL_d(q)_{v_1^P})^{\iota}$ stabilizes  $\langle v_2,\ldots,v_d\rangle$ and the group   $M_1\cong\SL_{d_1}(q)$ is transitive on $\langle v_2,\ldots,v_d\rangle \setminus \{0\}$, we conclude that $\Ome_1$ is an orbit of $(\SL_d(q)_{v_1^P})^{\iota}$.
Considering the actions of $N_1$ and $R_1$, one can see that $\Ome_2$ and $\Ome_3$ are also orbits of $(\SL_d(q)_{v_1^P})^{\iota}$. \qed
\end{proof}

\section{Proof of Theorem~\ref{Thm-2}}

We will prove Theorem~\ref{Thm-2} in this final section.

Let $\Ga$ be a connected $G$-arc-transitive bi-Cayley graph on a dihedral group $H=\l a\r:\l b\r\cong\D_{2n}$
with $G\ge H$,
and let $\alpha$ be a vertex of $\Ga$. 

\begin{lemma}\label{Ver-QP}
If $G$ is quasiprimitive on $V\Ga$, then
either $\Ga$ is the complete graph $\K_{4n}$, or $\Ga$  is isomorphic to one of the graphs in Table~\ref{VQ-gps}.
\end{lemma}

\begin{table}[htbp]
\caption{Vertex-quasiprimitive arc-transitive bi-dihedrants}\label{VQ-gps}
\[
\begin{array}{lllllll}
\hline
\text{Line} & (G,G_\a)  & \text{Valency} & \Aut\Ga & \Ga\\
\hline
1 & (\ZZ_{2}^4 : \mathrm{SO}_{4}^{+}(2),\mathrm{SO}_{4}^{+}(2)) & 6 &G& \mathrm{VO}^{+}_4(2)\\
2 & (\ZZ_{2}^4 : \mathrm{SO}_{4}^{+}(2),\mathrm{SO}_{4}^{+}(2)) & 9 &G&  \overline{\mathrm{VO}^{+}_4(2)}\\

3 & (\ZZ_{2}^4 :  \mathrm{SO}_{4}^{-}(2),\mathrm{SO}_{4}^{-}(2)) & 5  & G&   \mathrm{VO}^{-}_4(2)\\
4 & (\ZZ_{2}^4 :  \mathrm{SO}_{4}^{-}(2),\mathrm{SO}_{4}^{-}(2)) &  10 & G& \overline{\mathrm{VO}^{-}_4(2)}\\
5 & (\A_5,\ZZ_3) & 3 & \ZZ_2\times \A_5& \mathrm{F020A}   \\
6 & (\S_5 ,\ZZ_6) & 6 & \ZZ_2\times \S_5& \GG_{20}^{(1)}\\
7 & (\S_5,\S_3) & 6 & \ZZ_2\times \S_5& \GG_{20}^{(2)} \\
8 & (\S_5,\S_3),(\S_5 ,\ZZ_6) & 6 & \ZZ_2^{10}:\S_5&  \GG_{20}^{(3)}\\
9 & (\PSL_2(p^e).f,\ZZ_p^e : \ZZ_{ (p^e-1)/4}.f), & p^e & \geq \ZZ_2 \times \mathrm{ P\Sigma L}_2(p^e)& \mathcal{G}_{(2,p^e)}   \\
 & p^e \equiv 5(\mathrm{mod }~8),\ f\mid e &  &  &   \\
\hline
\end{array}
\]
\end{table}

\begin{proof}
First suppose that $G$ is primitive.
Then $(G,G_\a,H)$ satisfies part (1) of Theorem~\ref{Thm-1}.
If $G$ is 2-transitive, then $\Ga=\K_{4n}$ is a complete graph.
Assume $G$ is not 2-transitive. Then $(G,G_\a)=(\ZZ_{2}^4 : \mathrm{SO}_{4}^{+}(2),\mathrm{SO}_{4}^{+}(2))$ or
$(\ZZ_{2}^4:\mathrm{SO}_{4}^{-}(2),\mathrm{SO}_{4}^{-}(2))$.
By Example~\ref{ex:affinepolargraph},  $\Ga$ is the affine polar graphs $ \mathrm{VO}^{\pm}_4(2)$ or their complementary graphs.

Now suppose that $G$ is imprimitive.
Then $G$ satisfies part (2) of Theorem~\ref{Thm-1}.
If $(G,G_\a,n)=(\A_5,\mathbb{Z}_3,5)$, then by Example~\ref{ex:A5Z3}, $\Ga$ is the graph $\mathrm{F020A}$.
If $(G,G_\a,n)=(\S_5,\ZZ_6,5)$  by Lemma~\ref{ex:S5Z6S3}, $\Ga=\mathcal{G}_{20}^{(i)}$ for some $i\in \{1,2,3\}$.
If $(G,G_\a,n)=(\S_5,\mathrm{S}_3,5)$, then computation in {\sc Magma}~\cite{Magma} shows that $\Ga=\mathcal{G}_{20}^{(2)}$ or $\mathcal{G}_{20}^{(3)}$.

It remains to consider part (2)(ii) of Theorem~\ref{Thm-1}, that is, $G=\PSL_{2}(p^e).f\leq \mathrm{P\Sigma L}_2(p^e)$ with $f\mid e$
and $p^e \equiv 5\pmod{8}$, and $G_\alpha=\ZZ_p^e : \ZZ_{(p^e-1)/4}.f$, and $H $ is  in $\PSL_2(p^e)$, and $n=(p^e+1)/2$.
Recall the notations in Example~\ref{ex:2q}.
Notice that $\PSL_2(q)_{v_1^P}\cong \ZZ_p^e : \ZZ_{(p^e-1)/4} \cong \PSL_2(q)_{\alpha}$, and $\PSL_2(q)$ has only one conjugate class of subgroups isomorphic to $\ZZ_p^e : \ZZ_{(p^e-1)/4}$.
Hence the action of $\PSL_2(q)$ on $V\Ga$ is equivalent to that on $\Ome$.
Therefore, $\Ga \cong \mathcal{G}_{(2,q)}$ by Example~\ref{ex:2q}. \qed
\end{proof}

Next we treat the vertex biquasiprimitive case.
Suppose $G$ is biquasiprimitive on $V\Ga$.
Then $|V\Ga|=4n$, $\Ga$ is a bipartite graph, say with bipartitions $\Del_1$ and $\Del_2$.
Denote by $G^+$ the stabilizer of $G$ fixing $\Del_1$,
and $H^+=H\cap G^+$.
Then $G_\a=G^+_\a$ for $\a\in V\Ga$, $G^+$ is of index 2 in $G$
and $G=\l G^+,g\r$ for some $g\in G$ such that $g^2\in G^+$.

\begin{lemma}\label{unfaithful}
If $G^+$ acts unfaithfully on $\Del_1$ or $\Del_2$, then $\Ga=\K_{2n,2n}$.
\end{lemma}
\proof Let $K_i$ be the kernel of $G^+$ acting on $\Del_i$ for $i=1$ and $2$.
It is easy to show that $K_1^g=K_2$ and $K_2=K_1^g$,
and so $K_1\times K_2$ is a normal subgroup of $G$.
Notice that $K_1\times K_2$ fixes $\Del_1$ and $\Del_2$,
by the biquasiprimitivity of $G$, $\Del_1$ and $\Del_2$ are the orbits of
$K_1\times K_2$ on $V\Ga$, and thereby
$K_i$ is transitive on $\Del_{3-i}$.
This implies $\Ga=\K_{2n,2n}$ is a complete bipartite graph.\qed

\begin{lemma}\label{Prim}
If $G^+$ acts faithfully and primitively on $\Del_1$ and $\Del_2$, then one of the following holds:
\begin{itemize}
\item[(i)] $\Ga=\K_{2n,2n}-2n\K_2$.
\item[(ii)] $\Ga=B(\mathrm{PG}(d-1,q))$ or $B'(\mathrm{PG}(d-1,q))$ and  $2n=(q^d-1)/(q-1)$.
\end{itemize}
\end{lemma}
\proof Assume first $H\le G^+$.
Then $H$ is regular on $\Del_1$ and $\Del_2$, hence $G^+$ is primitive containing a regular dihedral subgroup $H$ on $\Del_1$ and $\Del_2$.
By Theorem~\ref{pri-d-gps}, $G^+$ is
$2$-transitive on $\Del_1$ and $\Del_2$,
and $(G^+,H,G_\a)$ is listed in Theorem~\ref{pri-d-gps}. 
If $(G^+)^{\Del_1}$  and $(G^+)^{\Del_2}$ are not permutation equivalent, then
$G^+$ has two different faithful $2$-transitive permutation representations,
and it follows from \cite[Notes 4]{Cameron} that either
$(G^+,G_{\a})=(\PGL(2,11),\A_5)$ or $\soc(G^+)=\PSL_d(q)$ with $d\ge 3$.
The former case is not possible as $|\Del_1|=|G^+:G_\a|=11\ne 2n=|H|$,
and the latter case is a contradiction as there is no candidate of groups in Theorem~\ref{pri-d-gps}
with socle $\PSL_d(q)$ and $d\ge 3$.
Thus $(G^+)^{\Del_1}$  and $(G^+)^{\Del_2}$ are permutation equivalent.
Now $G_{\a}$ fixes a vertex $\g\in\Del_2$ and is transitive on $\Del_2\setminus\{\g\}$.
It follows $\Ga=\K_{2n,2n}-2n\K_2$, as in part (i) of Lemma~\ref{Prim}.

Thus assume that $H$ is not contained in $G^+$ in the following. Then $G=HG^+$ and $H^+=H\cap G^+<H$.
Since $H/H^+=H/(H\cap G^+)\cong HG^+/G^+=G/G^+\cong\ZZ_2$,
$H^+$ is of index 2 in $H$,
hence either $H^+=\l a\r\cong\ZZ_n$ or $H^+=\l a^2\r:\l b\r\cong\D_n$ with $n$ even.
Notice that $H^+$ is biregular on $\Del_1$ and $\Del_2$, namely $G^+$ is a primitive group with a biregular cyclic subgroup $\l a\r$,
or a biregular dihedral subgroup $H^+=\l a^2\r:\l b\r\cong\D_n$.
Hence $(G^+,H^+)$ satisfies \cite[Theorem 3.3]{Muller13} or part (1) of Theorem~\ref{Thm-1}.
In particular, $G^+$ is affine or almost simple.

Suppose that $G^+$ is $2$-transitive on $\Del_i$.
If  $(G^+)^{\Del_1}$  and $(G^+)^{\Del_2}$ are permutation equivalent,
then $\Ga=\K_{n,n}-n\K_2$, as in part (i) of Lemma~\ref{Prim}.
Assume that $(G^+)^{\Del_1}$  and $(G^+)^{\Del_2}$  are not permutation equivalent.
Note that $(G^+,G_\a)\ne (\PSL_2(11),\A_5)$ by \cite[Theorem 3.3]{Muller13} or part (1) of Theorem~\ref{Thm-1}.
Therefore, by~\cite[Notes 4]{Cameron},  $\PGL_d(q)\le G\le\PGammaL_d(q)$ with $q\ge 3$ and $|\Del_i|=(q^d-1)/(q-1)$ is even,
this gives rise to $\Ga=B(\mathrm{PG}(d-1,q))$ or $B'(\mathrm{PG}(d-1,q))$, as in part (ii) of Lemma~\ref{Prim}.

Suppose now that $G^+$ is not 2-transitive on $\Del_i$.
Assume first that $G^+$ is affine.
By \cite[Theorem 3.3]{Muller13} and part (1) of Theorem~\ref{Thm-1},
the possibilities for $ (G^{+},(G^{+})_\alpha)$  are as follows.

\begin{itemize}
\item[(i)] $H^+\cong\ZZ_8$,  and $(G^{+},(G^{+})_\alpha)=(\AGL_3(2),\GL_3(2))$, as in part (1)(d) of \cite[Theorem 3.3]{Muller13};
\item[(ii)] $H^+=\D_8$, and    $(G^{+},(G^{+})_\alpha)=(\ZZ_{2}^4 : \mathrm{SO}^{+}_4(2),\mathrm{SO}^{+}_4(2))$, or $(\ZZ_{2}^4 : \mathrm{SO}^{-}_4(2),\mathrm{SO}^{-}_4(2))$.
\end{itemize}

Notice that $\Ga$ is  connected  $G$-symmetric and $G^{+}$-semisymmetric, and $ (G^{+})_\alpha=G_{\alpha}$, and $G=G^{+}.2$ lies in $\N_{\Aut\Ga}(G^{+})$.
For (i), computation in {\sc Magma} shows that there are two connected symmetric and $G^{+}$-semisymmetric graphs, whose valencies are $7$ and $8$.
If the valency is $7$, then $\Aut\Ga=\mathrm{S}_2 \times \mathrm{S}_8$, and any subgroup $L$ of $\N_{\Aut\Ga}(G^{+})$ such that $|L|/|G^{+}|=2$ and $L$ is transitive on $V\Ga$  is not biquasiprimitive on $V\Ga$.
If the valency is $8$, then we also have that any subgroup $L$ of $N_{\Aut\Ga}(G^{+})$ such that $|L|/|G^{+}|=2$ and $L$ is transitive on $V\Ga$  is not biquasiprimitive on $V\Ga$.
For (ii), if $(G^{+},(G^{+})_\alpha)=(\ZZ_{2}^4 : \mathrm{SO}^{+}_4(2),\mathrm{SO}^{+}_4(2))$, then computation in {\sc Magma} shows that there are two connected symmetric  and $G^{+}$-semisymmetric graphs, whose valencies are $6$ and $9$, while for both graphs, any subgroup $L$ of $\N_{\Aut\Ga}(G^{+})$ such that $|L|/|G^{+}|=2$ and $L$ is transitive on $V\Ga$  is not biquasiprimitive on $V\Ga$.
If   $(G^{+},(G^{+})_\alpha)=(\ZZ_{2}^4 : \mathrm{SO}^{-}_4(2),\mathrm{SO}^{-}_4(2))$, then computation in {\sc Magma} shows there are two connected symmetric  and $G^{+}$-semisymmetric graphs, whose valencies are $5$ and $10$, while for both graphs, any subgroup $L$ of $\N_{\Aut\Ga}(G^{+})$ such that $|L|/|G^{+}|=2$ and $L$ is transitive on $V\Ga$  is not biquasiprimitive on $V\Ga$.


Now assume that $G^+$ is almost simple and let $T=\soc(G^+)$.
Since $G=G^+.2$ is biquasiprimitive on $V\Ga$,
we have $\C_G(T)=1$, for otherwise $\C_G(T)\cong\ZZ_2$ has more then two orbits on $V\Ga$, a contradiction.
It follows that $G$ is also almost simple with socle $T$.
Consequently, we may set $G=T.o$ and $G^+=T.o_1$,
with $o_1<o\le\Out(T)$ and $|o:o_1|=2$; in particular $2$ divides $ |\Out(T)|$.
Checking the candidates in \cite[Theorem 3.3]{Muller13} and part (1) of Theorem~\ref{Thm-1}, the only possibility is:
$H^+=\ZZ_5$, $G^+=\A_5$, and $G=\S_5$ acts on the set of 2-subsets of $\{1,2,3,4,5\}$.
Computation in {\sc Magma} shows that there are two connected symmetric and $G^{+}$-semisymmetric graphs, whose valencies are $3$ and $6$.
Furthermore, for both of graphs, $\Aut\Ga=\mathrm{S}_2\times\mathrm{S}_5 $, and the subgroup $\mathrm{S}_5$ of  $ \Aut\Ga$  contains no bi-regular $\D_{10}$ on $V\Ga$, a contradiction.\qed

\begin{lemma}\label{Imprim}
If $G^+$ acts faithfully on $\Del_1$ and $\Del_2$,
and imprimitively on $\Del_1$ or $\Del_2$,
then
 \begin{itemize}
 \item[\rm (i)] $\PGL_d(q)\leq G^{+} \leq \PGammaL_d(q)$ and $G=G^{+} : \langle b\rangle \leq \Aut(\PSL_d(q))$, where $d\geq 3$ and $q$ is odd, and $b$ acts as an inverse-transpose automorphism of $\PSL_d(q)$;
 \item[\rm (ii)] $n=(q^{d}-1)/(q-1)$ and $G_\alpha$ is  the unique subgroup  with index $2$ of the stabilizer of a  $1$-dimensional subspace of $\mathbb{F}_q^d$ in $G^{+}$;
 \item[\rm (iii)]  $ \langle a\rangle$ is a Single subgroup of $\PGL_d(q)$;
  \item[\rm (iv)] $\Ga \cong \mathcal{G}_{(d,q)}^{(1)}$  or $ \mathcal{G}_{(d,q)}^{(2)}$ as Example~\ref{ex:d>3q}.
 \end{itemize}
\end{lemma}

\proof Suppose that $G^+$ is imprimitive on $\Del_i$ with $i=1$ or $2$.
Let $\BB=\{B_1,B_2,\cdots,B_m\}$ be a maximal $G^+$-invariant partition on $\Del_i$.
Then $\BB^g=\{B_1^g,B_2^g,\cdots,$
$B_m^g\}$ is a maximal $G^+$-block system on $\Del_{3-i}$.
Set $\CC=\BB\cup\BB^g$.
As $g^2\in G^+$ fixes $\BB$  setwise,
$\CC$ is a $G$-block system on $V\Ga$.
Since $G$ is biquasiprimitive on $V\Ga$ and $|\CC|=2|\BB|\ge 4$,
we derive that $G$ acts faithfully on $\CC$.
Denote by $H_{(\BB)}$ and $H_{(\BB^g)}$ the kernels of
$H$ acting on $\BB$ and $\BB^g$ respectively.
For $x\in H$, notice that
$x\in H_{(\BB)}$ if and only if $B_i^x=B_i$ for $1\le i\le m$,
or if and only if $(B_i^g)^{g^{-1}xg}=B_i^g$,
or equivalent $g^{-1}xg\in H_{(\BB^g)}$,
we conclude that $H_{(\BB^g)}=H_{(\BB)}^g$.
In particular, $|H_{(\BB)}|=|H_{(\BB^g)}|$.
If $|H_{(\BB)}|=2$, as $H_{(\BB)}\lhd H$, we have
$2\mid n$, notice that the dihedral group $\D_{2n}$ with $n$ even has unique normal subgroup of order 2,
we further obtain $H_{(\BB)}=H_{(\BB^g)}$.
Thereby $H$ acts unfaithfully on $\CC$, which is a contradiction as $G$ acts faithfully on $\CC$.
If $|H_{(\BB)}|>2$, by Lemma~\ref{Subgroup},
$H_{(\BB)}\cap H_{(\BB^g)}$ contains a nontrivial normal subgroup of $H$.
This means that $H$ acts unfaithfully on $\CC$,
which is a contradiction as $G$ acts faithfully on $\CC$.

Therefore $|H_{(\BB)}|=|H_{(\BB^g)}|=1$,
namely $H$ acts faithfully on $\BB$ and $\BB^g$,
or equivalently $H_B$ is a core-free subgroup of $H$ for each $B\in\CC$,
implying $H_B=1$ or $\l a^jb\r$ for some $j$.
Since $|\BB|<|\Del_1|=|H|$, $H$ is not regular on $\BB$,
so $H_B=\l a^jb\r\cong\ZZ_2$.
It then follows from  $\l a\r H_B=H$ that $\l a\r$ is transitive and so regular on $\BB$ and $\BB^g$.
In particular $|\BB|=n$, $|B|=2$ and $|G_B:G_\a|=2$.
Moreover, the maximality of $\BB$ and $\BB^g$ implies that
$(G^+)^{\BB}$ and $(G^+)^{\BB^g}$ are primitive,
hence both $(G^+)^{\BB}$ and $(G^+)^{\BB^g}$ satisfy \cite{Jones02}.

We claim that $G^+$ acts faithfully on $\BB$.
Suppose for a contradiction that $G^+$ acts unfaithfully on $\BB$. Let $G^+_{(\BB)}$ be the kernel of $G^+$ on $\BB$.
Then it is easy to show that $(G^+_{(\BB)})^g$ is the kernel of $G^+$ acting on $\BB^g$,
and $G^+_{(\BB)}\times (G^+_{(\BB)})^g\le G^+$ is normal in $G$.
By the biquasiprimitivity of $G$, $G^+_{(\BB)}\times (G^+_{(\BB)})^g\le G^+$ is transitive on $\Del_1$ and $\Del_2$,
and so transitive on $\BB$ and $\BB^g$,
it follows that $G^+_{(\BB)}$ is transitive on $\BB^g$,
and $(G^+_{(\BB)})^{\BB^g}$ is transitive on $\BB$.
Recall $|\BB|=n$, $|B|=2$, and $G^+$ acts faithfully on $\Del_i$,
we derive that $$G^+_{(\BB)}\lesssim (G^+)^{B_1}\times (G^+)^{B_2}\times\cdots\times (G^+)^{B_n}\lesssim\S_2^n$$
is soluble.
Recall that $(G^+)^{\BB^g}$ is affine or almost simple.
Since $1\ne G^+_{(\BB)}\cong (G^+_{(\BB)})^{\BB^g}$ is soluble  and normal in $(G^+)^{\BB^g}$,
$(G^+)^{\BB^g}$ is not almost simple, hence $(G^+)^{\BB}$ is affine. Further, by \cite{Jones02},
$\ZZ_p\lhd (G^+)^{\BB}\le\AGammaL(1,p)$, where $p=|\BB|=n$ is a prime.
However, this is  a contradiction, since
$(G^+_{(\BB)})^g$ is a $2$-group and is transitive on $\BB$.

We have shown that $G^+$ acts faithfully on $\BB$.
Then $(G^+)^{\BB}\cong G^+$,
and  as $G_B$ has a subgroup with index $2$, it follows from~\cite{Jones02} that
one of the following holds:
\begin{itemize}
\item[(i)] $(G^+,G_B)=(\M_{11},\A_6.2)$ or $(\S_n,\S_{n-1})$ with $n\ge 4$;
\item[(ii)] $\PGL_d(q)\le G^+\le\PGammaL_d(q)$ and $n= (q^d-1)/(q-1)$, and $d\geq 3$ or $d=2$
and $q\neq 8$, and $\langle a\rangle$ is a Singer subgroup.
\end{itemize}

If $G^+=\M_{11}$, as $\Out(\M_{11})=1$, we have $G=G^+.\ZZ_2\cong\M_{11}\times\ZZ_2$,
so $G$ has a normal subgroup isomorphic to $\ZZ_2$ having more than 4 orbits on $V\Ga$,
contradicting the bi-quasiprimitivity of $G$.
If $(G^+,G_B)=(\S_n,\S_{n-1})$, then $G_\a=\A_{n-1}<\soc(G^+)\cong\A_n$, it follows $\soc(G^+)$  has four orbits on $V\Ga$,
also contradicting the bi-quasiprimitivity of $G$.

Consider case (ii).
Since  $T:=\soc(G^{+})=\PSL_d(q)$ is characteristic in $G^{+}$ and $G^{+}$ is normal in $G$, we see that $T$ is normal in $G$, and so is  $ \C_G(T)$.
Notice that $\C_G(T)=1 $ or $\ZZ_2$.
Since $G$ acts biquasiprimitively on $V\Ga$, we see that $\C_G(T)=1$, and so $G$ is almost simple with socle $T$, and particularly, $G\leq \Aut(\PSL_d(q))$.
Let $L=\PGL_d(q)$ be the normal subgroup of $G$.
Since $G$ acts biquasiprimitively on $V\Ga$, we see that $L$ is transitive on $\Del_1$, which implies that $|G_\alpha|/|L_\alpha|=|G^{+}|/|L|$.
Recall that $G^{+}$ acts faithfully and primitively on $\mathcal{B}$.
Hence $L$ is transitive on $\mathcal{B}$, which implies that $|G_B|/|L_B|=|G^{+}|/|L|$.
It follows that $|G_\alpha|/|L_\alpha|=|G_B|/|L_B|=|G^{+}|/|L|$.
This together with $|G_B|/|G_\alpha|=2$ implies that $|L_B|/|L_\alpha|=2$.
Notice $L_{B}=[q^{d-1}] :  \GL_{d-1}(q)\cong \AGL_{d-1}(q)$.
If $q$ is even, then $L_{B}$ has no subgroup with index $2$, a contradiction.
Therefore, $q$ is odd.

Suppose that $H<\PGammaL_d(q)$.
Since $\langle a\rangle$ is a Singer subgroup, by Lemma~\ref{lm:mdq}(a) (corresponding to  $m=1$) we see that the case $d\geq 3$ is impossible, and so $d=2$.
Now  $L_B=\ZZ_p^e : \ZZ_{q-1}\cong \AGL_1(q)$.
From $|L_B|/|L_\alpha|=2$ we see that $L_\alpha=\ZZ_p^e : \ZZ_{(q-1)/2}$, which is a subgroup of $T=\PSL_2(q)$.
It follows that $T_\alpha=L_\alpha$, and hence $T$ has four orbits on $V\Ga$, contradicting that $G$ is bi-quasiprimitive on $V\Ga$.

Therefore, $H$ is not contained in $ \PGammaL_d(q)$.
Notice that   $\Aut(\PSL_2(q))=\PGammaL_2(q)$ and $\Aut(\PSL_d(q))=\PGammaL_d(q) : \langle \iota\rangle$  for $d\geq 3$, where $\iota$ is an inverse-transpose automorphism.
Since $\langle a \rangle \leq H\cap G^{+}$ and $G^{+}\leq \PGammaL_d(q)$, we conclude that $d\geq 3$ and $G=G^{+} : \langle b\rangle$.

Next, we show that $d$ is odd, and  $\Ga \cong \mathcal{G}_{(d,q)}^{(1)}$ or $\mathcal{G}_{(d,q)}^{(2)}$.
Recall that $G^{+}$ acts faithfully and primitively on both $\mathcal{B}$ and $\mathcal{B}^g$, and $|\mathcal{B}|=n= (q^d-1)/(q-1)$.
Since $G$ contains an inverse-transpose automorphism, it follows that $\Gamma_{\mathcal{C}}=B(\mathrm{PG}(d-1,q))$ or $B'(\mathrm{PG}(d-1,q))$.
Particularly,  we may view $\mathcal{C}=\mathcal{B} \cup \mathcal{B}^g$ as the union of $1$- and $(d-1)$-subspaces  of a $d$-dimensional vector space $V $ over $\mathbb{F}_q$.
Fix a basis $(v_1,v_2,\ldots,v_d)$ of $V$.
By the transitivity of $G$ on $\mathcal{C}$, we may view $B$ as the $1$-subspace $\langle v_1 \rangle$.
Since $L_B \cong \AGL_{d-1}(q)$ and $|L_B|/|L_\alpha|=2$, we conclude that $L_B'=\mathrm{ASL}_{d-1}(q)$ is contained in $ L_\alpha $.
Then, from $L_B/L_B'\cong \ZZ_{q-1}$, we see that $L_\alpha$ is the unique subgroup with index $2$ in $L_B$.
We identify $L$ with $\GL_d(q)/Z$, where $Z$ is the center of $\GL_d(q)$.
Notice that $Z< \GL_d(q)_{\langle v_1 \rangle}$.
Therefore,  the full preimage of $L_\alpha$ in $\GL_d(q)$, namely $ \hat{\,}L_\alpha$, is the subgroup with index $2$ in $\GL_d(q)_{\langle v_1 \rangle}$ that contains $Z$.

Now we identify this group $ \hat{\,}L_\alpha$ in $\GL_d(q)_{\langle v_1 \rangle}$.
Let $\mu$ be a generator of $\mathbb{F}_q^{*}$.
Then $ \GL_d(q)_{\langle v_1 \rangle}=N:(M \times R_2)= N:(M_1:(R_3\times R_2)) $, where
\begin{align*}
N=&\left\{
 \begin{pmatrix}
1 & \textbf{0} \\
\textbf{a}^{\mathsf{T}} & I_{n-1}\\
\end{pmatrix}  : I_{n-1}=\mathrm{diag}(1,\ldots,1) \in \GL_{d-1}(q),\ \textbf{a}  \in \mathbb{F}_{q}^{d-1} \right\} \cong  [q^{d-1}],\\
M=&\left\{
 \begin{pmatrix}
1 & \textbf{0} \\
\textbf{0}^{\mathsf{T}} & \textbf{b}\\
\end{pmatrix}  : \textbf{b} \in \GL_{d-1}(q) \right\} \cong \GL_{d-1}(q),\\
M_1=&\left\{
 \begin{pmatrix}
1 & \textbf{0} \\
\textbf{0}^{\mathsf{T}} & \textbf{b}\\
\end{pmatrix}  : \textbf{b} \in \SL_{d-1}(q) \right\} \cong \SL_{d-1}(q),\\
R_2=&\left\{
 \begin{pmatrix}
\lambda & \textbf{0} \\
\textbf{0}^{\mathsf{T}} & I_{n-1} \\
\end{pmatrix}    :  I_{n-1}=\mathrm{diag}(1,\ldots,1) \in \GL_{d-1}(q), \lambda \in \langle \mu \rangle \right\} \cong \ZZ_{q-1},\\
R_3=&\left\{
 \begin{pmatrix}
1 & \textbf{0} \\
\textbf{0}^{\mathsf{T}} & \textbf{c}   \\
\end{pmatrix}   :  \textbf{c}=\mathrm{diag}(\lambda,\ldots,1) \in \GL_{d-1}(q), \lambda \in \langle \mu \rangle \right\} \cong \ZZ_{q-1}.
\end{align*}
From $\GL_d(q)_{\langle v_1 \rangle}'=NM_1\cong \mathrm{ASL}_{d-1}(q)$ we conclude $NM_1 < \hat{\,}L_\alpha$.
Let $x=\mathrm{diag}\{\mu,1,\ldots,1 \}\in R_2$.
Suppose that $x \notin \hat{\,}L_\alpha $.
Then from $\GL_d(q)_{\langle v_1 \rangle}=N:(M \times R_2)$ we conclude that $\hat{\,}L_\alpha=N:(M \times \langle x^2\rangle)$.
From $\mathrm{diag}\{1,\mu^{-1}, \ldots, \mu^{-1} \} \in M \leq   \hat{\,}L_\alpha$ and $\mathrm{diag}\{\mu,\mu, \ldots, \mu \} \in Z \leq   \hat{\,}L_\alpha$, we see $x \in \hat{\,}L_\alpha$, a contradiction.
Therefore, $x \in  \hat{\,}L_\alpha$.
Let $y=\mathrm{diag}\{1,\mu,1,\ldots,1 \}\in R_3 $.
Then $R_2=\langle x\rangle$ and $ R_3=\langle y\rangle $.
Since $\GL_d(q)_{\langle v_1 \rangle}/NM_1\cong R_2\times R_3\cong \ZZ_{q-1}^2$ and $\GL_d(q)_{\langle v_1 \rangle}/\hat{\,}L_\alpha \cong \ZZ_2$, it follows that $y \notin \hat{\,}L_\alpha$.
Therefore, $\hat{\,}L_\alpha=N:(M_1: \langle y^2\rangle \times R_2)$, where the group $M_1: \langle y^2\rangle <M \cong \GL_{d-1}(q)$.

Consider the homomorphism $ \sigma$ from $M_1: \langle y^2 \rangle$ to $\mathbb{F}_q^{*}$ given by the determinant.
Since the kernel of $\sigma$ is $M_1\cong \SL_{d-1}(q)$, we see that $\sigma(M_1: \langle y^2 \rangle)\cong \langle y^2 \rangle \cong \langle \mu^2\rangle$.
From $x=\mathrm{diag}\{\mu,1,\ldots,1 \} \in \hat{\,}L_\alpha$ and $\mathrm{diag}\{\mu^{-1},\mu^{-1},\ldots,\mu^{-1} \} \in Z \leq  \hat{\,}L_\alpha$, we have $z:=\mathrm{diag}\{1,\mu^{-1},\ldots,\mu^{-1} \} \in Z$ and hence $z \in M_1: \langle y^2 \rangle$.
Then $\sigma(z)=\mu^{-(d-1)} \in \langle \mu^2\rangle$, which implies $\mu^{d-1} \in \langle \mu^2\rangle$.
If $d$ is even, then $d-1$ is odd, contradicting $\mu^{d-1} \in \langle \mu^2\rangle$.

Recall the notations of Example~\ref{ex:d>3q}.
We have seen in Lemma~\ref{lm:cGdq} that $\GL_d(q)_{v_1^P}=N:(M\times R)\cong [q^{d-1}]:(\GL_{d-1}(q)\times \ZZ_{(q-1)/2})$, and hence $(\ZZ_2\times \PGL_d(q))_{v_1^P}\cong \GL_d(q)_{v_1^P}/P\cong \AGL_{d-1}(q)$.
Since $\PSL_d(q)$ is transitive on $\Ome$ by Lemma~\ref{lm:cGdq}, so are $\PGL_d(q)$ and $\ZZ_2\times \PGL_d(q)$.
Hence $\PGL_d(q)_{v_1^P}$ is the unique subgroup with index $2$ in $(\ZZ_2\times \PGL_d(q))_{v_1^P}\cong \AGL_{d-1}(q)$.
Recall that $L_\alpha=\PGL_d(q)_\alpha$ is also the unique  subgroup with index $2$ in $L_B\cong \AGL_{d-1}(q)$.
Therefore, the action of $\PGL_d(q)$ on $\Ome$ is equivalent to that on $\Del_1$.
Then,  $\Ga \cong \mathcal{G}_{(d,q)}^{(1)}$ or $\mathcal{G}_{(d,q)}^{(2)}$ by Example~\ref{ex:d>3q}. \qed

Now we are ready to complete the proof of Theorem~\ref{Thm-2}.

\vskip0.1in
\noindent{\bf{Proof of Theorem~\ref{Thm-2}.}}
Let $\Ga$ be an $X$-arc-transitive bi-Cayley graph on a dihedral group $H=\l a\r:\l b\r\cong\D_{2n}$ with $n\ge 3$.
Suppose that $N$ is a normal subgroup of $X$ which is a maximal with respect to having at least three orbits on $V\Ga$.
By \cite[Theorem~3.3]{DGJ22}, $\Ga$ is a normal $r$-cover of $\Ga_N$ with $r$ is a divisor of the valency of $\Ga$.
Furthermore, $X/N\leq\Aut\Ga$, $\Ga_N$ is $X/N$-arc-transitive, and $X/N$ is either quasiprimitive or bi-quasiprimitive on $V(\Ga_N)$.

Set $\bar{X}=X/N$, and $Y=NH$ and $\bar{Y}=Y/N$. Then $\bar{Y}=Y/N\cong H/(H\cap N)$. Since $H$ is dihedral, it follows that $H/(H\cap N)\cong\ZZ_2$ or $\D_{2k}$ with $k\mid n$.

Let $\BB=\{B_1,B_2,\dots,B_t\}=V(\Ga_N)$, namely, the set of $N$-orbits on $V\Ga$,
and let $H_0$ and $H_1$ be the two orbits of $H$ on $V\Ga$.
Then $\BB$ is an $X$-invariant partition of $V\Ga$, and by Lemma~\ref{Block},
either \begin{itemize}
\item[(a)] each $B\in\BB$ is a subset of $H_0$ or $H_1$; or
\item[(b)] each $B\in\BB$ satisfies $B\cap H_0\ne \emptyset$ and $B\cap H_1\ne\emptyset$.
\end{itemize}

Assume first case (a) occurs. Then $H_0$ and $H_1$ are the two orbits of $Y$ on $V\Ga$,
and so $\bar{Y}$ has two orbits $(H_0)_N$ and $(H_1)_N$ on $V(\Ga_N)$,
where $(H_0)_N$ and $(H_1)_N$ denote the sets of $N$-orbits on $H_0$ and $H_1$ respectively.


Recall that $\bar{Y}\cong\ZZ_2$ or $\D_{2k}$ with $k\mid n$. For the former, we have $|\BB|=2$ and then $\G_N\cong\K_{2,2}$, which is a graph appearing in part (1)(a). 
For the latter, we may let $\bar{Y}:=\l \bar{a}\r:\l \bar{b}\r$, where $\bar{a}=aN$ and $\bar{b}=bN$.
Notice that $\bar{Y}_B$ is core-free in $\bar{Y}$,
so $\bar{Y}_B=1$ or $\l \bar{a}^s\bar{b}\r\cong\ZZ_2$ for some integer $s$.
In the former case, $\bar{Y}$ is regular on $(H_0)_N$ and $(H_1)_N$.
So $\Ga_N$ is a bi-Cayley graph of $\bar{Y}$, and then by Lemmas~\ref{Ver-QP}--\ref{Imprim},
$\Ga$ is one of the graphs in part (1) of Theorem~\ref{Thm-2}.
In the latter case, we have $\l \bar{a}\r \bar{Y}_B=\bar{Y}$, and
so $\l \bar{a}\r$ is transitive and hence regular on each $(H_i)_N$. It follows that $\Ga_N$ is a bi-Cayley graph of $\l a\r$, and so $\Ga_N$ is
a quasiprimitive or bi-quasiprimitive bi-circulant, as in part (2) of Theorem~\ref{Thm-2}.

Assume now case (b) occurs. Then $H$ is transitive on $\BB$, and so
$\bar{Y}$ is transitive on $\BB$. Since $|\BB|=t\geq3$, we have $\bar{Y}\ncong\ZZ_2$, and so $\bar{Y}=\l \bar{a}\r:\l \bar{b}\r\cong\D_{2k}$ with $k\mid n$. Then $\bar{Y}_B=1$ or $\l \bar{a}^s\bar{b}\r\cong\ZZ_2$ for some integer $s$.
For the former, $\Ga_N$ is a arc-transitive dihedrant,
satisfying part (3) of Theorem~\ref{Thm-2}.
and for the latter, $\l \bar{a}\r$ is regular on $\BB$, so
$\Ga_N$ is a circulant, satisfying part (2) of Theorem~\ref{Thm-2}.
This completes the proof.\qed

\end{document}